\documentclass[11pt,a4paper]{article}

\usepackage[margin=1in]{geometry}
\usepackage[T1]{fontenc}
\usepackage{lmodern}
\usepackage{microtype}
\usepackage[numbers,sort&compress]{natbib}
\usepackage{siunitx}
\usepackage{booktabs}
\usepackage{amsthm}
\usepackage{subcaption}
\usepackage{graphicx}
\usepackage{amsmath,amssymb,bm}
\usepackage{accents}
\usepackage{algorithm,algpseudocode}
\usepackage{algorithmicx}
\usepackage[colorlinks=true,linkcolor=blue,citecolor=blue,urlcolor=blue]{hyperref}

\algrenewcommand\alglinenumber[1]{\footnotesize #1:}
\newcommand{\algFontSize}{\small}

\newtheorem{theorem}{Theorem}[section]

\newtheorem{lemma}{Lemma}[section]

\newtheorem{corollary}{Corollary}[section]

\newtheorem{remark}{Remark}[section]

\numberwithin{equation}{section}

\title{An unfitted boundary algebraic equation method with Calder\'on preconditioning for 2D Stokes flow in irregular geometry}

\author{Wenjun Ying\thanks{School of Mathematical Sciences, MOE-LSC and Institute of Natural Sciences, Shanghai Jiao Tong University, Minhang, Shanghai 200240, China. Email: \texttt{wying@sjtu.edu.cn}}
\and
Qing Xia\thanks{Corresponding author. College of Science, Mathematics and Technology, Wenzhou Kean University, Zhejiang 325060, China. Email: \texttt{qxia@kean.edu}}}

\date{}

\begin{document}
\maketitle

\begin{abstract}
We present an unfitted boundary algebraic equation method for the two-dimensional exterior/interior Stokes equations on a staggered MAC grid. By constructing an explicit free-space pair of velocity and pressure lattice Green's functions (LGFs) from free-space Laplace LGFs, we represent homogeneous fields using sources supported exclusively on thin staggered boundary layers. This formulation imposes physical Dirichlet data at cut points via local interpolation, while sampled-normal rank updates remove hydrostatic null modes associated with single or multiple obstacles. The workflow parallels that of classical boundary integral formulations and requires no artificial boundary conditions for exterior flows, but follows a discretize-then-represent route and does not require singular/near-singular quadrature. The resulting dense boundary system is solved via GMRES, utilizing a componentwise discrete Calder\'on preconditioner built from the scalar Laplace kernel and padded FFTs for fast volume convolutions. Extensive numerical validation, including multiply connected domains, narrow gaps, and Moffatt eddies, confirms discrete incompressibility to solver accuracy and recovers the expected Moffatt eddy scaling. We achieve second-order velocity and pressure convergence and bound maximum discrete divergence within numerical accuracy. The discrete Calder\'on preconditioner reduces the condition number by orders of magnitude and yields nearly mesh-independent conditioning in exterior configurations, while remaining effective---though more demanding---for narrow-gap and fine-grid interior problems.

\medskip
\noindent\textit{Keywords:} Stokes equations; Lattice Green's function; boundary algebraic equation; discrete potential theory; MAC scheme; Calder\'on preconditioning
\end{abstract}

\section{Introduction}\label{sec:intro}

The steady incompressible Stokes equations describe viscous flow in
the low-Reynolds-number regime and arise in particulate suspensions,
microfluidics, porous media, and fluid-structure interaction. In these
applications, the fluid region may contain curved inclusions, multiple
obstacles, narrow gaps, or moving boundaries. Boundary-fitted discretizations
can resolve such geometries accurately, but generating and updating a fitted
mesh may dominate the computational workflow. Unfitted approaches instead
embed the physical domain in a simple background grid. Representative examples
include immersed boundary and immersed interface methods
\citep{peskin2002immersed,leveque1994immersed,tan2009fast}, fictitious-domain
methods \citep{glowinski1994fictitious,burman2014fictitious}, and cut finite
elements \citep{hansbo2014cut}. A central difficulty is to impose the physical
boundary conditions accurately without sacrificing the stability and discrete
incompressibility of the underlying flow discretization.

Boundary integral formulations offer a complementary viewpoint. By
representing a Stokes field with fundamental solutions, they reduce a domain
problem to unknown densities on the boundary and are well suited to exterior
and multiple-particle flows \citep{pozrikidis1992boundary}. Fast multipole and
related acceleration techniques can reduce the cost of evaluating the
resulting layer potentials \citep{tornberg2008fast}. Their implementation requires accurate singular and nearly singular quadrature on the
physical boundary. 

Kernel-free boundary integral methods move the potential
evaluation to an auxiliary Cartesian grid while retaining a boundary
formulation \citep{ying2007kernel,ying2013kernel,zhou2024correction}.
More precisely, a kernel-free boundary integral (KFBI) method begins
with a continuous boundary integral equation but evaluates its boundary and
volume potentials indirectly by solving corrected interface problems on an
auxiliary structured grid. The analytical Green's kernel is therefore not
needed. This strategy has recently been extended to Stokes--Brinkman interface
problems with a corrected MAC discretization \citep{zhou2026brinkman}. 

The method developed here is related to both classical BIE and KFBI, but is best described
as a fully discrete boundary-integral analogue rather than as a direct
discretization of the continuous Stokes BIE. A classical BIE follows a
represent-then-discretize route, whereas the present method follows a
discretize-then-represent route. 
Like immersed-boundary, KFBI and cut-cell schemes, it works on a simple background grid and incorporates geometry only near the physical boundary. Like boundary-integral methods, it reduces the unknown to a boundary density and evaluates the field through potentials. The distinction is the order of construction: by first committing to a discrete Stokes operator and then building its potentials, one obtains a formulation that inherits the structure of the underlying MAC discretization, while retaining the boundary-centric organization familiar from classical potential theory. In particular, because the kernels are discrete from the outset, no singular or nearly singular quadratures are required.

This construction is rooted in discrete potential theory
\citep{duffin1953discrete} and in boundary algebraic
equations for lattice problems \citep{martinsson2009boundary}. Closely related
difference-potential methods characterize admissible boundary traces through
an auxiliary discrete problem \citep{ryaben2012method}. More recent unfitted
formulations combine discrete potentials with local basis functions to impose
boundary data on implicitly defined geometries
\citep{xia2023local,xia2026geom}. Lattice Green's functions are especially
attractive in this setting because they invert a translation-invariant
difference operator on the infinite lattice and therefore permit fast
convolution on a Cartesian grid and retain far-field asymptotics, thus requiring no artificial boundary conditions 
\citep{martinsson2002asymptotic,gillman2014fast,dorschner2020fast}.
Related LGF constructions also arise outside fluid solvers: Hodapp et~al.\ \citep{hodapp2019lattice} develop lattice Green function methods as a discrete boundary-element analogue for atomistic/continuum coupling in solid mechanics, while Bhamidipati et~al.\ \citep{bhamidipati2022inverse} use lattice Green functions to formulate inverse source problems for localizing defects in conducting lattices and metamaterials.
The physical
geometry enters only through a thin set of lattice nodes near the boundary and
through local interpolation to the boundary itself.

LGF-based incompressible flow solvers already exploit mimetic
staggered-grid operators and fast discrete convolutions on unbounded domains
\citep{liska2016fast}. In the work of Liska and Colonius, the scalar Laplace LGF
is used to solve the discrete Poisson subproblems arising from a projection
method. Their immersed-boundary extension couples prescribed surfaces to this
flow solver through regularized interpolation and spreading, with boundary
forces acting as Lagrange multipliers \citep{liska2017immersed}. They do not form
the velocity--pressure LGF pair of the coupled steady Stokes operator as a
discrete layer-potential kernel. Discrete Green's functions for Stokes response
have also been developed for correcting the self-induced velocity in
Euler--Lagrange point-particle simulations, including channel geometries
\citep{horwitz2022discrete}. They exploited periodicity in two homogeneous directions to tabulate wall-dependent discrete Stokes velocity responses for a three-dimensional two-plane channel. Their kernels depend parametrically on wall-normal source position and are used locally for particle self-disturbance correction. In contrast, the present velocity and pressure LGFs are translation-invariant free-space kernels of the MAC Stokes operator and are used globally to construct boundary and volume potentials.

The specific advance here is a translation-invariant free-space,
matrix-valued velocity LGF and associated pressure kernel for the steady
two-dimensional MAC Stokes operator. Constructed from regularized Laplace and
biharmonic LGFs, this pair supplies homogeneous boundary-layer and inhomogeneous
volume potentials. Physical Dirichlet data are imposed sharply at cut points,
yielding a rank-completed, Calder\'on-preconditioned boundary algebraic equation.

Extending this idea from a scalar elliptic equation \citep{xia2026geom} to the Stokes
system introduces several coupled difficulties. The two velocity components
and the pressure occupy different locations on the staggered MAC grid; the
velocity kernel is tensor valued; the representation must preserve the
discrete divergence constraint; and pressure-jump modes generate hydrostatic
nullspaces whose dimension changes with the number of disconnected objects.
Moreover, the boundary-density equation is dense and becomes increasingly ill
conditioned under refinement. The purpose of this work is to address these
issues in a single unfitted boundary algebraic framework.

The main contributions are as follows.
\begin{enumerate}
\item We derive explicit free-space velocity and pressure lattice
Green's functions for the two-dimensional steady MAC Stokes discretization
from regularized Laplace and biharmonic lattice Green's functions. In contrast
to using a scalar Laplace LGF only within a projection solve, the resulting
kernel pair represents the coupled discrete Stokes response, satisfies the
momentum equations and divergence constraint, and provides both homogeneous
layer potentials and inhomogeneous volume potentials.
\item We formulate a Stokes boundary algebraic equation on componentwise
staggered boundary layers. Local cut-point interpolation imposes Dirichlet
data, while sampled-normal rank completion removes one independent hydrostatic
mode for each disconnected obstacle.
\item We construct a componentwise discrete Calder\'on preconditioner from the
scalar Laplace lattice Green's function, including rank updates for the
constant modes on multiple objects. Nonzero body forces are handled by a
particular volume potential without changing the boundary operator, and padded
FFTs provide efficient lattice convolutions.
\item We evaluate the method on smooth and multiply connected interior flows,
a narrow-gap geometry, a body-forced problem, corner eddies, and exterior flows
with one and several obstacles. The tests assess accuracy, conditioning,
iteration counts, topology, and preservation of discrete incompressibility.
\end{enumerate}

The remainder of the paper is organized as follows.
Section~\ref{sec:stokes_lgf} derives and regularizes the MAC Stokes lattice
Green's functions. Section~\ref{sec:bae} develops the boundary algebraic
representation, and Section~\ref{sec:boundary-closure} imposes the physical
Dirichlet data and completes the hydrostatic nullspace. The discrete
Calder\'on preconditioner is introduced in
Section~\ref{sec:laplace-calderon-preconditioning}. Inhomogeneous body forcing
and FFT acceleration are treated in Sections~\ref{sec:inhomogeneous_forcing}
and \ref{sec:fft_acceleration}, respectively, and
Section~\ref{sec:end-to-end-algorithm} summarizes the complete computational
procedure. Section~\ref{sec:numerical-experiments} presents the numerical
experiments, followed by conclusions in Section~\ref{sec:conclusion}.

\section{MAC Stokes lattice Green's functions}
\label{sec:stokes_lgf}

We first recall the lattice Green's function for the two-dimensional
discrete Laplacian, since the Stokes lattice Green's function can be
written in terms of the harmonic/biharmonic lattice Green's function. Let
$m=(m_1,m_2)\in\mathbb Z^2$, and define the unscaled five-point
difference operator
\begin{align}
    (L\phi)(m)
    =
    \phi(m+e_1)+\phi(m-e_1)+\phi(m+e_2)+\phi(m-e_2)-4\phi(m).
\end{align}
Its Fourier symbol is
\begin{align}
    \lambda(\theta)
    =
    2\cos\theta_1+2\cos\theta_2-4,
    \qquad
    \theta=(\theta_1,\theta_2)\in[-\pi,\pi]^2 .
\end{align}
Formally, the Laplace lattice Green's function is
\begin{align}
    G(m)
    =
    \frac{1}{4\pi^2}
    \operatorname{f.p.}
    \int_{-\pi}^{\pi}\int_{-\pi}^{\pi}
    \frac{e^{i(m_1\theta_1+m_2\theta_2)}}{\lambda(\theta)}
    \,d\theta_1\,d\theta_2 ,
    \label{eq:laplace-lgf-formal}
\end{align}
where the integral is interpreted in the finite-part sense because
$\lambda(\theta)\sim-|\theta|^2$ near $\theta=0$. Equivalently, one may use
the normalized regularized representative
\begin{align}
    G_{\rm reg}(m)
    =
    \frac{1}{4\pi^2}
    \int_{-\pi}^{\pi}\int_{-\pi}^{\pi}
    \frac{e^{i(m_1\theta_1+m_2\theta_2)}-1}{\lambda(\theta)}
    \,d\theta_1\,d\theta_2 .
    \label{eq:laplace-lgf-regularized}
\end{align}
This fixes the additive constant by setting $G_{\rm reg}(0)=0$. The asymptotic properties and efficient computation of $G_{\rm reg}$ are well established \citep{martinsson2002asymptotic,borwein2013lattice,xia2026geom}.

The corresponding formal biharmonic lattice Green's function is
\begin{align}
    H(m)
    =
    \frac{1}{4\pi^2}
    \operatorname{f.p.}
    \int_{-\pi}^{\pi}\int_{-\pi}^{\pi}
    \frac{e^{i(m_1\theta_1+m_2\theta_2)}}{\lambda(\theta)^2}
    \,d\theta_1\,d\theta_2 .
    \label{eq:biharmonic-lgf-formal}
\end{align}

The singularity in \eqref{eq:biharmonic-lgf-formal} is stronger than that
of the Laplace kernel. Thus direct numerical evaluation of
\eqref{eq:biharmonic-lgf-formal} is inconvenient. In computation, we use a
regularized, gauge-fixed representative of $H$, obtained from the
regularized Laplace LGF. The Fourier-symbol derivation below is performed
with the formal kernels \eqref{eq:laplace-lgf-formal} and
\eqref{eq:biharmonic-lgf-formal}; the numerical implementation replaces
them by the chosen regularized representatives. This distinction fixes only
the pressure and velocity reference gauges and does not change the discrete
Stokes equations satisfied by the resulting kernels.

We consider the steady incompressible Stokes equations
\begin{subequations}
\begin{align}
    -\mu \Delta \bm u+\nabla p &= \bm f, \\
    \nabla\cdot \bm u &=0,
\end{align}
\end{subequations}
where $\bm u=(u,v)$, $p$ is the pressure, $\mu>0$ is the viscosity, and
$\bm f=(f^u,f^v)$ is the applied force. We discretize the equations by the
standard MAC scheme on a uniform grid of spacing $h$ (Figure~\ref{fig:mac_points}). 

\begin{figure}[htbp]
\centering
\includegraphics[width=0.35\textwidth]{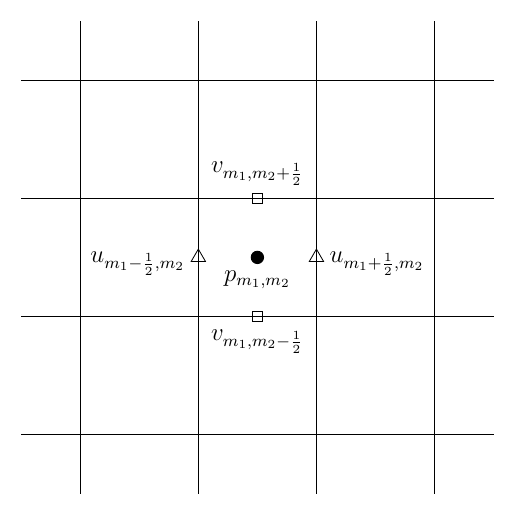}
\caption{MAC scheme}\label{fig:mac_points}
\end{figure}

The horizontal velocity
$u$ is stored at vertical cell faces,
\begin{align}
    u_{m_1+1/2,m_2}
    \approx
    u(x_{m_1+1/2},y_{m_2}),
\end{align}
the vertical velocity $v$ is stored at horizontal cell faces,
\begin{align}
    v_{m_1,m_2+1/2}
    \approx
    v(x_{m_1},y_{m_2+1/2}),
\end{align}
and the pressure is stored at cell centers,
\begin{align}
    p_{m_1,m_2}
    \approx
    p(x_{m_1},y_{m_2}).
\end{align}

Multiplying the momentum equations by $h^2$, the MAC discretization is
\begin{subequations}
\begin{align}
-\mu L u_{m_1+1/2,m_2}
+h\left(p_{m_1+1,m_2}-p_{m_1,m_2}\right)
&=
h^2 f^u_{m_1+1/2,m_2},
\label{eq:mac-u}
\\
-\mu L v_{m_1,m_2+1/2}
+h\left(p_{m_1,m_2+1}-p_{m_1,m_2}\right)
&=
h^2 f^v_{m_1,m_2+1/2},
\label{eq:mac-v}
\\
u_{m_1+1/2,m_2}-u_{m_1-1/2,m_2}
+
v_{m_1,m_2+1/2}-v_{m_1,m_2-1/2}
&=0 .
\label{eq:mac-div}
\end{align}
\end{subequations}
Here $L$ denotes the same unscaled five-point difference operator applied on
the corresponding velocity grid.

Define the unscaled difference operators
\begin{subequations}
\begin{align}
    D_x^+ \phi(m_1,m_2)
    &= \phi(m_1+1,m_2)-\phi(m_1,m_2), \\
    D_y^+ \phi(m_1,m_2)
    &= \phi(m_1,m_2+1)-\phi(m_1,m_2), \\
    D_x^- \phi(m_1,m_2)
    &= \phi(m_1,m_2)-\phi(m_1-1,m_2), \\
    D_y^- \phi(m_1,m_2)
    &= \phi(m_1,m_2)-\phi(m_1,m_2-1),
\end{align}
\end{subequations}
and
\begin{subequations}
\begin{align}
    D_{xx}\phi(m_1,m_2)
    &=
    \phi(m_1+1,m_2)-2\phi(m_1,m_2)+\phi(m_1-1,m_2), \\
    D_{yy}\phi(m_1,m_2)
    &=
    \phi(m_1,m_2+1)-2\phi(m_1,m_2)+\phi(m_1,m_2-1).
\end{align}
\end{subequations}

\begin{theorem}[MAC Stokes lattice Green's function]
\label{thm:mac-stokes-lgf}
Let $G$ be the Laplace lattice Green's function and let $H$ be a
biharmonic lattice Green's function satisfying $LH=G$ in the
distributional sense. Then the MAC Stokes
lattice Green's function is
\begin{align}
\mathcal S(m_1,m_2)
=
-\frac{1}{\mu}
\begin{bmatrix}
D_{yy}H(m_1,m_2)
&
- D_y^+D_x^+H(m_1,m_2-1)
\\[4pt]
- D_y^+D_x^+H(m_1-1,m_2)
&
D_{xx}H(m_1,m_2)
\end{bmatrix}.
\label{eq:stokes-lgf}
\end{align}
The associated pressure lattice Green's function is
\begin{align}
\mathcal P(m_1,m_2)
=
\frac{1}{h}
\begin{bmatrix}
D_x^-G(m_1,m_2)
\\[3pt]
D_y^-G(m_1,m_2)
\end{bmatrix}.
\label{eq:pressure-lgf}
\end{align}
The first column of $\mathcal S$ gives the velocity generated by a unit force
in the $x$-direction at a horizontal velocity node, and the second column
gives the velocity generated by a unit force in the $y$-direction at a
vertical velocity node.
\end{theorem}

\begin{proof}
We give the derivation for a unit force in the $x$-direction. The derivation
for a unit force in the $y$-direction is identical after interchanging $x$
and $y$.

Use the staggered Fourier transforms
\begin{subequations}
\begin{align}
    \widehat u(\theta)
    &=
    h^2\sum_{m\in\mathbb Z^2}
    u_{m_1+1/2,m_2}
    e^{-i((m_1+1/2)\theta_1+m_2\theta_2)}, \\
    \widehat v(\theta)
    &=
    h^2\sum_{m\in\mathbb Z^2}
    v_{m_1,m_2+1/2}
    e^{-i(m_1\theta_1+(m_2+1/2)\theta_2)}, \\
    \widehat p(\theta)
    &=
    h^2\sum_{m\in\mathbb Z^2}
    p_{m_1,m_2}
    e^{-i(m_1\theta_1+m_2\theta_2)} .
\end{align}
\end{subequations}
Let
\begin{align}
    s_j=\sin(\theta_j/2),
    \qquad
    \gamma_j=e^{i\theta_j/2}-e^{-i\theta_j/2}=2i s_j,
    \qquad j=1,2 .
\end{align}
For a unit $x$-force at the horizontal velocity node $(1/2,0)$, the
transformed MAC equations become
\begin{subequations}
\begin{align}
    -\mu\lambda(\theta)\widehat u
    +h\gamma_1\widehat p
    &=
    h^2 e^{-i\theta_1/2},
    \label{eq:fourier-u-xforce}
    \\
    -\mu\lambda(\theta)\widehat v
    +h\gamma_2\widehat p
    &=
    0,
    \label{eq:fourier-v-xforce}
    \\
    \gamma_1\widehat u+\gamma_2\widehat v
    &=0 .
    \label{eq:fourier-div}
\end{align}
\end{subequations}
Since
\begin{align}
    \lambda(\theta)
    =
    -4(s_1^2+s_2^2),
\end{align}
the incompressibility constraint projects the forcing onto the discrete
divergence-free subspace. Solving
\eqref{eq:fourier-u-xforce}--\eqref{eq:fourier-div} gives
\begin{subequations}
\begin{align}
    \widehat p^{(1)}
    &=
    \frac{ih}{2}
    \frac{4s_1 e^{-i\theta_1/2}}{\lambda(\theta)},
    \label{eq:p-hat-xforce}
    \\
    \widehat u^{(1)}
    &=
    \frac{h^2 e^{-i\theta_1/2}}{\mu}
    \frac{4s_2^2}{\lambda(\theta)^2},
    \label{eq:u-hat-xforce}
    \\
    \widehat v^{(1)}
    &=
    -\frac{h^2 e^{-i\theta_1/2}}{\mu}
    \frac{4s_1s_2}{\lambda(\theta)^2}.
    \label{eq:v-hat-xforce}
\end{align}
\end{subequations}
The superscript $(1)$ denotes forcing in the $x$-direction.

We now identify the inverse transforms. Since
\begin{align}
    D_{yy}H
    \quad\longleftrightarrow\quad
    \left(e^{i\theta_2}+e^{-i\theta_2}-2\right)
    \frac{1}{\lambda(\theta)^2}
    =
    -\frac{4s_2^2}{\lambda(\theta)^2},
\end{align}
equation \eqref{eq:u-hat-xforce} gives
\begin{align}
    u^{(1)}_{m_1+1/2,m_2}
    =
    -\frac{1}{\mu}
    D_{yy}H(m_1,m_2).
\end{align}
Similarly,
\begin{align}
    D_y^+D_x^+H(m_1-1,m_2)
    \quad\longleftrightarrow\quad
    -\frac{4s_1s_2 e^{-i\theta_1/2}e^{i\theta_2/2}}{\lambda(\theta)^2},
\end{align}
with the staggered inverse phase for $v$, so that
\begin{align}
    v^{(1)}_{m_1,m_2+1/2}
    =
    \frac{1}{\mu}
    D_y^+D_x^+H(m_1-1,m_2).
\end{align}
The pressure follows from \eqref{eq:p-hat-xforce}:
\begin{align}
    p^{(1)}_{m_1,m_2}
    =
    \frac{1}{h}
    D_x^-G(m_1,m_2).
\end{align}

For a unit force in the $y$-direction, the same calculation yields
\begin{subequations}
\begin{align}
    \widehat p^{(2)}
    &=
    \frac{ih}{2}
    \frac{4s_2 e^{-i\theta_2/2}}{\lambda(\theta)}, \\
    \widehat u^{(2)}
    &=
    -\frac{h^2 e^{-i\theta_2/2}}{\mu}
    \frac{4s_1s_2}{\lambda(\theta)^2}, \\
    \widehat v^{(2)}
    &=
    \frac{h^2 e^{-i\theta_2/2}}{\mu}
    \frac{4s_1^2}{\lambda(\theta)^2}.
\end{align}
\end{subequations}
Transforming back gives
\begin{subequations}
\begin{align}
    u^{(2)}_{m_1+1/2,m_2}
    &=
    \frac{1}{\mu}
    D_y^+D_x^+H(m_1,m_2-1), \\
    v^{(2)}_{m_1,m_2+1/2}
    &=
    -\frac{1}{\mu}
    D_{xx}H(m_1,m_2), \\
    p^{(2)}_{m_1,m_2}
    &=
    \frac{1}{h}
    D_y^-G(m_1,m_2).
\end{align}
\end{subequations}
Collecting the two velocity responses as the columns of $\mathcal S$, and
collecting the two pressure responses as $\mathcal P$, gives
\eqref{eq:stokes-lgf} and \eqref{eq:pressure-lgf}.
\end{proof}

\begin{remark}[Regularization and gauge]
The proof above uses the Fourier multipliers $1/\lambda$ and
$1/\lambda^2$ in the finite-part sense. In the actual computation, however,
one should not evaluate \eqref{eq:biharmonic-lgf-formal} directly. Instead,
a regularized representative of $H$ is used. Different regularizations
correspond to different choices of additive discrete harmonic or biharmonic
null-space components. These choices fix the pressure and velocity gauges.
Once a gauge is fixed, the formulas \eqref{eq:stokes-lgf} and
\eqref{eq:pressure-lgf} define the corresponding discrete Stokes kernel. The
same gauge must be used consistently throughout the computation.
\end{remark}

Although $H$ satisfies
\begin{align}
    LH=G
\end{align}
on the infinite lattice, direct computation of $H$ from this equation is
inconvenient. A finite-domain truncation requires artificial boundary
conditions for a non-decaying biharmonic kernel, and direct quadrature of
\eqref{eq:biharmonic-lgf-formal} is strongly singular at the zero Fourier
mode.

Instead, we compute a regularized representative of $H$ from the regularized
Laplace LGF. The Duffin--Shelly identity \citep{DuffinShelly1958Polyharmonic} gives
\begin{align}
16H(m_1,m_2)
={}&
m_1^2\Bigl(
2G(m_1,m_2)
+G(m_1,m_2+1)
+G(m_1,m_2-1)
\Bigr)
\nonumber\\
&+
m_2^2\Bigl(
2G(m_1,m_2)
+G(m_1+1,m_2)
+G(m_1-1,m_2)
\Bigr)
\nonumber\\
&-
2G(m_1,m_2).
\label{eq:duffin-shelly-H}
\end{align}
Here $G$ should be understood as the same regularized Laplace LGF used in
\eqref{eq:laplace-lgf-regularized}. Equation
\eqref{eq:duffin-shelly-H} fixes a compatible regularization of $H$. This is
the representative used in the numerical construction of the Stokes LGF
through \eqref{eq:stokes-lgf}.
In computation, the regularized kernels $G$ and $H$ are tabulated for lattice indices ranging from $0$ to $N$ in each coordinate and extended by symmetry as needed.

\section{Stokes boundary algebraic equations}\label{sec:bae}
We now use the regularized velocity and pressure LGFs to represent
discrete Stokes fields generated by sources on a thin lattice layer. Throughout
this section, $\mathcal S$ and $\mathcal P$ denote the gauge-fixed kernels
constructed in Section~\ref{sec:stokes_lgf}.

\subsection{LGF representation of discrete Stokes solutions}
\begin{theorem}\label{thm:rep1}
Let $C=\bigcup_n C_n$ be a union of cells indexed by their centers
$n$. The following velocity and pressure potentials satisfy the homogeneous
Stokes difference equations at every lattice point whose stencil does not
intersect $C$:
\begin{align}
\bm{u}(m)=
\left[
\begin{array}{c}
u(m+e_1/2)\\
v(m+e_2/2)
\end{array}
\right] = \sum_n\mathcal{S}(m,n)\left[
\begin{array}{c}
q^u_{n+e_1/2}\\
q^v_{n+e_2/2}\\
\end{array}
\right]
\end{align}
and \begin{align}
p(m) = \sum_n\mathcal{P}(m,n)\left[
\begin{array}{c}
q^u_{n+e_1/2}\\
q^v_{n+e_2/2}\\
\end{array}
\right]
\end{align}
Here $m=(m_1,m_2)$, $n=(n_1,n_2)$, $e_1=(1,0)$, and
$e_2=(0,1)$. The tensors $\mathcal S$ and $\mathcal P$ map the staggered
source components to velocity and pressure, respectively.
\end{theorem}

\begin{proof}
In the operator form, the velocity and pressure tensors satisfy
\begin{subequations}
\begin{align}
-\mu \Delta_h \mathcal{S}(m) + I\nabla_h\mathcal{P}(m) &= \frac{1}{h^2}\delta_{m}I,\\
\nabla_h \cdot \mathcal{S}(m) &= \bm{0}.
\end{align}
\end{subequations}
For the velocity equations,
\begin{subequations}
\begin{align}
&-\mu\Delta_h \bm{u}(m)+\nabla_h p(m) \\
=&-\mu \Delta_h \sum_n\mathcal{S}(m,n)\left[
\begin{array}{c}
q^u_{n+e_1/2}\\
q^v_{n+e_2/2}\\
\end{array}
\right]+\nabla_h \sum_n\mathcal{P}(m,n)\left[
\begin{array}{c}
q^u_{n+e_1/2}\\
q^v_{n+e_2/2}\\
\end{array}
\right] \\
=& \sum_n \left(-\mu\Delta_h\mathcal{S}(m,n)+I\nabla_h \mathcal{P} \right)\left[
\begin{array}{c}
q^u_{n+e_1/2}\\
q^v_{n+e_2/2}\\
\end{array}
\right]\\
=& \frac{1}{h^2}\sum_n \left[
\begin{array}{c}
\delta_{m-n}q^u_{n+e_1/2}\\
\delta_{m-n}q^v_{n+e_2/2}\\
\end{array}
\right]
\end{align}
\end{subequations}
which vanishes whenever $m\notin C$.

The discrete incompressibility condition is satisfied identically:
\begin{subequations}
\begin{align}
\nabla_h \cdot\bm{u}(m) &= \nabla_h \cdot \sum_n\mathcal{S}(m,n)\left[
\begin{array}{c}
q^u_{n+e_1/2}\\
q^v_{n+e_2/2}\\
\end{array}
\right]\\
&=\sum_n\nabla_h \cdot\mathcal{S}(m,n)\left[
\begin{array}{c}
q^u_{n+e_1/2}\\
q^v_{n+e_2/2}\\
\end{array}
\right]=\bm{0}
\end{align}
\end{subequations}
for any densities.
\end{proof}

\subsection{Staggered-grid point sets}
We next define the lattice sets used to separate interior and exterior
stencils. We will focus on a bounded domain $\Omega$, and follow the difference-potentials
convention of \citet{ryaben2012method}. (For an exterior problem, the roles of the
interior and exterior sets are reversed.)

Embed $\Omega$ in the infinite lattice $h\mathbb Z^2$ with spacing
$h$. The sets $M_+^u$ and $M_+^v$ contain the
$u$- and $v$-velocity nodes inside $\Omega$, whereas $M_-^u$ and $M_-^v$
contain the corresponding exterior nodes.

At a cell-centered lattice point $(m_1h,m_2h)$, the standard
five-point stencil is
\begin{align}
N^5_{m_1,m_2} = \{(m_1h,m_2h),((m_1\pm 1)h,m_2h),(m_1h,(m_2 \pm 1)h)\}.
\end{align}

Taking the union of these stencils over the $u$-velocity nodes gives
\begin{align}
N^u_\pm =
\bigcup_{((m_1+1/2)h,m_2h)\in M^u_\pm}
N^5_{m_1+1/2,m_2}.
\end{align}
Their intersection $\gamma^u:=N_+^u\cap N_-^u$ is the discrete
interface layer for the $u$ component. We further split it into the interior
and exterior layers $\gamma_+^u:=\gamma^u\cap M_+^u$ and
$\gamma_-^u:=\gamma^u\cap M_-^u$ (Figure~\ref{fig:gamma_u}).

\begin{figure}[htbp]
    \centering
    \begin{subfigure}{0.35\textwidth}
        \centering
        \includegraphics[width=\textwidth]{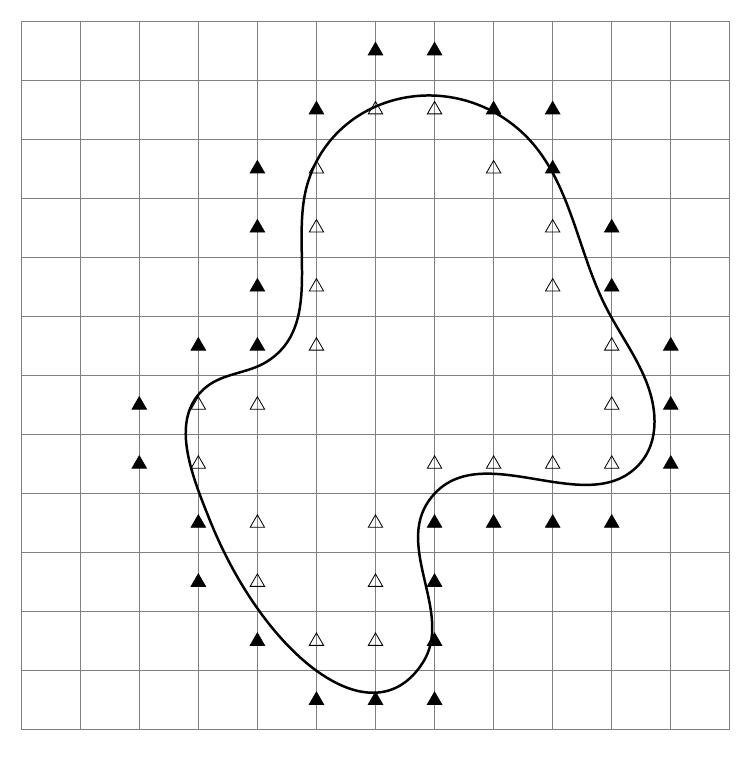}
        \caption{$\gamma^u_\pm$}
        \label{fig:gamma_u}
    \end{subfigure}
    ~
    \begin{subfigure}{0.35\textwidth}
        \centering
        \includegraphics[width=\textwidth]{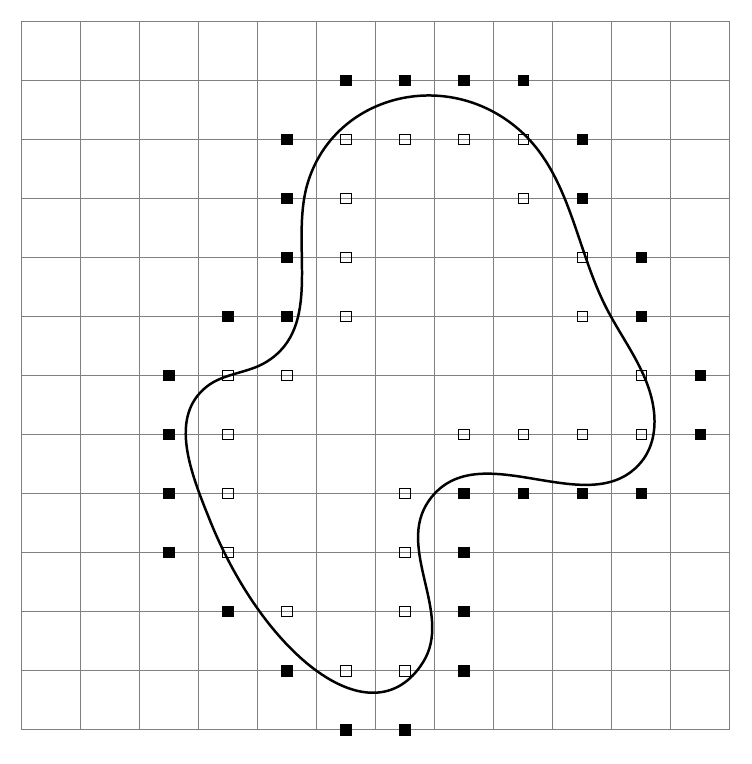}
        \caption{$\gamma^v_\pm$}
        \label{fig:gamma_v}
    \end{subfigure}
    \caption{Illustration of $\gamma^u_\pm$ and $\gamma^v_\pm$ (solid triangles and squares denote $\gamma^u_-$ and $\gamma^v_-$ and hollow triangles and squares denote $\gamma^u_+$ and $\gamma^v_+$)}
    \label{fig:gamma}
\end{figure}

The sets $M_\pm^v$, $N_\pm^v$, $\gamma^v$, and
$\gamma_\pm^v$ (Figure~\ref{fig:gamma_v}) are defined analogously after replacing the $u$-grid location
$(m_1+1/2,m_2)$ by the $v$-grid location $(m_1,m_2+1/2)$.

Analogous pressure sets $M_\pm^p$ may be introduced, but no
independent pressure-layer density is required: the pressure is recovered from
the two velocity-source densities through the pressure LGF $\mathcal P$.

\subsection{Boundary algebraic trace equations}
Theorem~\ref{thm:rep1} pairs $q^u$ and $q^v$ on the two staggered
faces of a common cell. Near an unfitted boundary, however,
$|\gamma^u|$ and $|\gamma^v|$ need not agree. The next result therefore
allows the two densities to be supported independently on their respective
exterior layers.

\begin{corollary}\label{prop:rep2}
Given the exterior boundary points $\gamma^u_-$ and $\gamma^v_-$ and the densities $q^u$ and $q^v$ defined at these points, the velocities and associated pressures
\begin{align}
\bm{u}^u(m)= \sum_{n+e_1/2\in \gamma^u_-}\mathcal{S}(m,n)\left[
\begin{array}{c}
q^u_{n+e_1/2}\\
0\\
\end{array}
\right],\quad
p^u(m) = \sum_{n+e_1/2\in \gamma^u_-}\mathcal{P}(m,n)\left[
\begin{array}{c}
q^u_{n+e_1/2}\\
0\\
\end{array}
\right]
\end{align}
and
\begin{align}
\bm{u}^v(m)= \sum_{n+e_2/2\in \gamma^v_-}\mathcal{S}(m,n)\left[
\begin{array}{c}
0\\
q^v_{n+e_2/2}\\
\end{array}
\right],\quad
p^v(m) = \sum_{n+e_2/2\in \gamma^v_-}\mathcal{P}(m,n)\left[
\begin{array}{c}
0\\
q^v_{n+e_2/2}\\
\end{array}
\right]
\end{align}
satisfy the Stokes difference equations
\eqref{eq:mac-u}--\eqref{eq:mac-div}; by linearity, so does their sum
$(\bm u^u+\bm u^v,p^u+p^v)$.
\end{corollary}

\begin{proof}
The densities $q^u$ and $q^v$ are arbitrary in the proof of Theorem~\ref{thm:rep1}, thus the pairs $(\bm{u}^u,p^u)$ and $(\bm{u}^v,p^v)$ satisfy the Stokes difference equations, respectively. Since the operators are linear, the additions also satisfy the difference equations.
\end{proof}

In particular, restricting the target points to $\gamma^u_+$ and $\gamma^u_-$ leads to
\begin{subequations}
\begin{align}
u_{\gamma^u_-} &= \sum_{n+e_1/2\in \gamma^u_-}\mathcal{S}^{-}_{11}(m,n)q^u_{n+e_1/2} + \sum_{n+e_2/2\in \gamma^v_-}\mathcal{S}^{-}_{12}(m,n)q^v_{n+e_2/2},\\
u_{\gamma^u_+} &= \sum_{n+e_1/2\in \gamma^u_-}\mathcal{S}^{+}_{11}(m,n)q^u_{n+e_1/2} + \sum_{n+e_2/2\in \gamma^v_-}\mathcal{S}^{+}_{12}(m,n)q^v_{n+e_2/2},
\end{align}
\end{subequations}
Here $\mathcal S^-$ denotes evaluation from the exterior source
layers to the exterior target layers, whereas $\mathcal S^+$ denotes
evaluation from the same sources to the interior target layers.

Similarly, restricting target points to $\gamma^v_\pm$ gives
\begin{subequations}
\begin{align}
v_{\gamma^v_-} &= \sum_{n+e_1/2\in \gamma^u_-}\mathcal{S}^{-}_{21}(m,n)q^u_{n+e_1/2} + \sum_{n+e_2/2\in \gamma^v_-}\mathcal{S}^{-}_{22}(m,n)q^v_{n+e_2/2},\\
v_{\gamma^v_+} &= \sum_{n+e_1/2\in \gamma^u_-}\mathcal{S}^{+}_{21}(m,n)q^u_{n+e_1/2} + \sum_{n+e_2/2\in \gamma^v_-}\mathcal{S}^{+}_{22}(m,n)q^v_{n+e_2/2}.
\end{align}
\end{subequations}

In compact block-matrix form, the mappings from the unknown density vector $\bm{q} = [q^u, q^v]^T$ to the velocity boundary traces are:
\begin{align}
\left(
\begin{array}{c}
u_{\gamma^u_-}\\
v_{\gamma^v_-}\\
\end{array}\right)
=\left[\begin{array}{cc}
\mathcal S^-_{11} & \mathcal S^-_{12}\\
\mathcal S^-_{21} & \mathcal S^-_{22}\\
\end{array}
\right]\left(
\begin{array}{c}
q^u_{n+e_1/2}\\
q^v_{n+e_2/2}
\end{array}
\right)\Rightarrow \bm{u}_{-} = S_-\bm{q}
\end{align}
and
\begin{align}
\left(
\begin{array}{c}
u_{\gamma^u_+}\\
v_{\gamma^v_+}\\
\end{array}\right)
=\left[\begin{array}{cc}
\mathcal S^+_{11} & \mathcal S^+_{12}\\
\mathcal S^+_{21} & \mathcal S^+_{22}\\
\end{array}
\right]\left(
\begin{array}{c}
q^u_{n+e_1/2}\\
q^v_{n+e_2/2}
\end{array}
\right)\Rightarrow \bm{u}_{+} = S_+\bm{q}
\end{align}
where $\bm u_\pm$ are the velocity traces on
$(\gamma_\pm^u,\gamma_\pm^v)$, respectively, and $S_\pm$ are the block
matrices assembled from $\mathcal S^\pm$.


\subsection{Nullspace and constrained invertibility}

\begin{lemma}[Energy identity for the staggered Stokes single layer]
\label{lem:stokes-single-layer-energy}
Let \(\bm q=(q^u,q^v)\) be a density supported on
\(\gamma_-:=\gamma_-^u\cup\gamma_-^v\), extended by zero to the whole
infinite staggered lattice. Let
\[
    \bm u = \mathcal S \bm q,\qquad p=\mathcal P \bm q
\]
be the corresponding regularized Stokes LGF potential. Assume that \(\bm q\)
satisfies the force-balance conditions
\begin{align}
    \sum_{\gamma_-^u} q^u=0,
    \qquad
    \sum_{\gamma_-^v} q^v =0,
\end{align}
so that the net force in each component vanishes and the associated velocity
field has finite discrete energy. Then
\begin{align}
    \langle \bm q,S_-\bm q\rangle_{\gamma_-}
    =
    \mu \|\nabla_h \bm u\|_{\ell^2(\mathbb Z^2)}^2
    \ge 0 .
    \label{eq:stokes-layer-energy}
\end{align}
In particular, \(S_-\) is positive semidefinite on the force-balanced
density space.
\end{lemma}

\begin{proof}
Since \(\bm q\) is supported on \(\gamma_-\), and since
\(\bm u|_{\gamma_-}=S_-\bm q\), we have
\begin{align}
    \langle \bm q,S_-\bm q\rangle_{\gamma_-}
    =
    \langle \bm q,\bm u\rangle_{\mathbb Z^2}.
\end{align}
The Stokes LGF potential satisfies
\begin{align}
    -\mu\Delta_h \bm u+\nabla_h p = \bm q,
    \qquad
    \nabla_h\cdot \bm u=0
\end{align}
on the infinite staggered lattice. Therefore
\begin{align}
    \langle \bm q,\bm u\rangle_{\mathbb Z^2}
    &=
    \langle -\mu\Delta_h \bm u+\nabla_h p,\bm u\rangle_{\mathbb Z^2}  \nonumber\\
    &=
    \mu \|\nabla_h \bm u\|_{\ell^2(\mathbb Z^2)}^2
    -
    \langle p,\nabla_h\cdot \bm u\rangle_{\mathbb Z^2}.
\end{align}
The boundary term at infinity vanishes because the density is
force-balanced and the regularized velocity field has finite energy. Since
\(\nabla_h\cdot \bm u=0\), the pressure term vanishes. Hence
\begin{align}
    \langle \bm q,S_-\bm q\rangle_{\gamma_-}
    =
    \mu \|\nabla_h \bm u\|_{\ell^2(\mathbb Z^2)}^2
    \ge 0 .
\end{align}
\end{proof}




\begin{theorem}[Nullspace of \(S_-\)]
\label{thm:Sminus-nullspace-pressure-jump}
Let \(S_-\) be the exterior-layer Stokes single-layer matrix assembled from
the regularized Stokes LGF. Assume that the density space is restricted to
the force-balanced class so that the corresponding velocity potential has
finite discrete energy. Let
\[
    Z_h
    :=
    \left\{
    \left.\nabla_h \chi\right|_{\gamma_-}
    :
    \chi \text{ is piecewise constant away from } \gamma_-
    \right\}
\]
be the space of discrete hydrostatic pressure-jump modes. Then
\[
    Z_h \subseteq \ker S_- .
\]
Moreover, if the only finite-energy homogeneous Stokes velocity field on
the infinite lattice with zero trace on \(\gamma_-\) is the zero velocity
field, then
\[
    \ker S_- = Z_h .
\]
\end{theorem}

\begin{proof}
We first show that \(Z_h\subseteq \ker S_-\). Let \(\bm q\in Z_h\). Then
\(\bm q=\nabla_h \chi\) on \(\gamma_-\), where \(\chi\) is piecewise constant
away from the boundary layer. Extend \(q\) by zero away from \(\gamma_-\).
The pair
\[
    \bm u\equiv 0,\qquad p=\chi
\]
satisfies
\[
    -\mu\Delta_h \bm u+\nabla_h p = \bm q,
    \qquad
    \nabla_h\cdot \bm u=0,
\]
in the distributional lattice sense. Therefore the velocity trace generated
by this density is zero on \(\gamma_-\), and hence
\[
    S_-\bm q=0.
\]
Thus \(Z_h\subseteq \ker S_-\).

Conversely, suppose \(S_-\bm q=0\), and let \((\bm u,p)\) be the Stokes LGF
potential generated by \(\bm q\). Since \(\bm q\) is supported on \(\gamma_-\),
\[
    \bm q^T S_- \bm q
    =
    \sum_{m\in\gamma_-} \bm q(m)\cdot \bm u(m).
\]
Using the discrete Stokes equations and summation by parts gives
\[
    \bm q^T S_- \bm q
    =
    \mu \|\nabla_h \bm u\|_{\ell_h^2}^2
    -
    \langle p,\nabla_h\cdot \bm u\rangle_h .
\]
The pressure term vanishes because \(\nabla_h\cdot \bm u=0\). Since
\(S_-\bm q=0\), the left-hand side is zero, and therefore
\[
    \nabla_h \bm u=0.
\]
The finite-energy condition, or equivalently the chosen velocity gauge,
eliminates nonzero constant velocities. Hence \(\bm u\equiv0\).

Substituting \(\bm u=0\) into the Stokes equations gives
\[
    \nabla_h p=\bm q
    \quad\text{on } \gamma_-,
    \qquad
    \nabla_h p=0
    \quad\text{away from } \gamma_-.
\]
Thus \(p\) is piecewise constant away from the boundary layer, and \(\bm q\) is
the discrete gradient of a piecewise constant pressure field. Therefore
\(\bm q\in Z_h\), which proves
\[
    \ker S_- \subseteq Z_h.
\]
Combining the two inclusions gives
\[
    \ker S_- = Z_h .
\]
\end{proof}

\begin{corollary}[Injectivity modulo hydrostatic modes]
\label{cor:Sminus-injective-modulo}
Assume the hypotheses of Theorem~\ref{thm:Sminus-nullspace-pressure-jump}.
Let
\[
    Q_h^0
    :=
    \left\{
    \bm q:\;
    \sum_{\gamma_-}\bm q=0,
    \quad
    \langle \bm q,z\rangle_{\gamma_-}=0
    \ \text{for all } z\in Z_h
    \right\}.
\]
Then
\[
    S_-\bm q=0,\qquad \bm q\in Q_h^0
    \quad\Longrightarrow\quad
    \bm q=0.
\]
Equivalently, \(S_-\) is injective on the constrained space \(Q_h^0\), or,
in quotient notation, \(S_-\) is injective on \(Q_h/Z_h\).
\end{corollary}

\begin{proof}
Let \(\bm q\in Q_h^0\) and suppose \(S_-\bm q=0\). By
Theorem~\ref{thm:Sminus-nullspace-pressure-jump},
\[
    \bm q\in \ker S_- = Z_h.
\]
But \(\bm q\in Q_h^0\) also implies
\[
    \langle \bm q,z\rangle_{\gamma_-}=0
    \qquad
    \forall z\in Z_h.
\]
Taking \(z=\bm q\) gives
\[
    \|\bm q\|_{\gamma_-}^2=0.
\]
Therefore \(\bm q=0\).
\end{proof}

\section{Dirichlet boundary closure}
\label{sec:boundary-closure}

No boundary condition has entered the construction so far.  The LGF
representation gives a discrete Stokes field satisfying the homogeneous
difference equations away from the source layer.  The purpose of the
boundary closure is to determine the unknown layer density
    $\bm q=
        \begin{bmatrix}
        q^u\\
        q^v
        \end{bmatrix}$
from the prescribed boundary data.

We consider the Dirichlet condition
\begin{align}
    u|_\Gamma = g^u,
    \qquad
    v|_\Gamma = g^v,
\end{align}
where $\Gamma$ is the continuous boundary.
Following the local boundary interpolation approach of
\citep{xia2023local,banks2016galerkin}, we impose the boundary condition by
evaluating local tensor-product hat functions at cut points.  For a grid
point \((x_j,y_k)\) of the appropriate staggered lattice, define
\begin{align}
    \phi_{jk}(x,y)
    =
    \phi\left(\frac{x-x_j}{h}\right)
    \phi\left(\frac{y-y_k}{h}\right),
\end{align}
where
\begin{align}
\phi(\xi)
=
\begin{cases}
1+\xi, & -1\leq \xi\leq 0,\\
1-\xi, & 0\leq \xi\leq 1,\\
0, & \text{otherwise}.
\end{cases}
\end{align}
The point \((x_j,y_k)\) is taken from either the \(x\)-velocity staggered
grid or the \(y\)-velocity staggered grid, depending on the velocity
component being interpolated.

\begin{figure}[htbp]
    \centering
    \begin{subfigure}{0.35\textwidth}
        \centering
        \includegraphics[width=\textwidth]{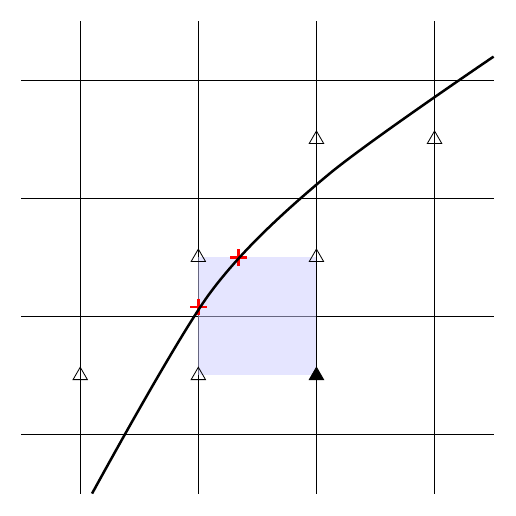}
        \caption{\(u\)-velocity half grids}
        \label{fig:cut1}
    \end{subfigure}
    ~
    \begin{subfigure}{0.35\textwidth}
        \centering
        \includegraphics[width=\textwidth]{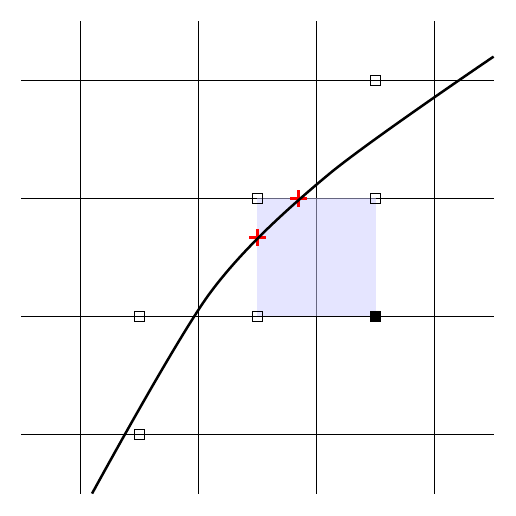}
        \caption{\(v\)-velocity half grids}
        \label{fig:cut2}
    \end{subfigure}
    \caption{Cut-cell interpolation for staggered velocity components. (Contribution of solid triangles/squares is zero at the intersections.)}
    \label{fig:cuts}
\end{figure}

For each exterior boundary-layer point in \(\gamma_-^u\) and
\(\gamma_-^v\), we select the nearest intersection point between the
corresponding staggered grid line and the physical boundary \(\Gamma\) (see Figure~\ref{fig:cuts}).  The
intersection points are denoted by
\begin{align}
    \bm x^u=\{x_a^u\}_{a=1}^{|\gamma_-^u|},
    \qquad
    \bm x^v=\{x_b^v\}_{b=1}^{|\gamma_-^v|}.
\end{align}
Thus
\begin{align}
    |\bm x^u|=|\gamma_-^u|,
    \qquad
    |\bm x^v|=|\gamma_-^v|.
\end{align}

For the \(u\)-component, we use the local representation
\begin{align}
    u(x,y)
    =
    \sum_{(m_1+1/2,m_2)\in\gamma^u}
    u_{m_1+1/2,m_2}\,
    \phi^u_{m_1+1/2,m_2}(x,y),
    \qquad
    (x,y)\in C_{\rm cut},
    \label{eq:u-local-interpolation}
\end{align}
where
\[
    \gamma^u=\gamma_+^u\cup\gamma_-^u .
\]
For the basic closure, the basis functions associated with
\(\gamma^u\) are sufficient at the selected intersection points.  Therefore
the \(u\)-component boundary condition is discretized as
\begin{align}
    \sum_{(m_1+1/2,m_2)\in\gamma^u}
    u_{m_1+1/2,m_2}\,
    \phi^u_{m_1+1/2,m_2}(x_a^u)
    =
    g^u(x_a^u),
    \qquad
    a=1,\ldots,|\gamma_-^u|.
\end{align}
In matrix form,
\begin{align}
    \Phi_+^u u_+ + \Phi_-^u u_- = g^u .
    \label{eq:u-bc-basic}
\end{align}
Similarly, for the \(v\)-component,
\begin{align}
    \sum_{(m_1,m_2+1/2)\in\gamma^v}
    v_{m_1,m_2+1/2}\,
    \phi^v_{m_1,m_2+1/2}(x_b^v)
    =
    g^v(x_b^v),
    \qquad
    b=1,\ldots,|\gamma_-^v|,
\end{align}
or
\begin{align}
    \Phi_+^v v_+ + \Phi_-^v v_- = g^v .
    \label{eq:v-bc-basic}
\end{align}
Combining \eqref{eq:u-bc-basic} and \eqref{eq:v-bc-basic} gives
\begin{align}
    \Phi_+\bm u_+ + \Phi_-\bm u_- = \bm g,
    \label{eq:combined-bc-basic}
\end{align}
where
\begin{align}
    \Phi_\pm
    =
    \begin{bmatrix}
    \Phi_\pm^u & 0\\
    0 & \Phi_\pm^v
    \end{bmatrix},
    \qquad
    \bm g=
    \begin{bmatrix}
    g^u\\
    g^v
    \end{bmatrix}.
\end{align}

The layer potential representation gives
\begin{align}
    \bm u_+ = S_+ \bm q,
    \qquad
    \bm u_- = S_- \bm q,
    \label{eq:upm-density}
\end{align}
with
\begin{align}
    S_\pm
    =
    \begin{bmatrix}
    S_{11}^\pm & S_{12}^\pm\\
    S_{21}^\pm & S_{22}^\pm
    \end{bmatrix}.
\end{align}
Substitution of \eqref{eq:upm-density} into
\eqref{eq:combined-bc-basic} gives the boundary-density equation
\begin{align}
    A \bm q = \bm g,
    \qquad
    A:=\Phi_+S_+ + \Phi_-S_- .
    \label{eq:density-bc-basic}
\end{align}

\paragraph{Rank-one completion for a single object}
For a single connected object, the hydrostatic nullspace is one-dimensional.
It is represented, up to discretization and scaling, by the normal density.
Let \(\nu=(\nu_1,\nu_2)\) be the outward unit normal to the continuous boundary \(\Gamma\) (pointing to fluids).  Define
the sampled normal vector on the intersection points by
\begin{align}
    n_\Gamma
    =
    \begin{bmatrix}
    \nu_1(x_1^u),
    \cdots,
    \nu_1(x_{|\gamma_-^u|}^u),
    \nu_2(x_1^v),
    \cdots,
    \nu_2(x_{|\gamma_-^v|}^v)
    \end{bmatrix}^T.
    \label{eq:sampled-normal-vector}
\end{align}

The rank-one completed boundary-density system is
\begin{align}
    A_c \bm q = \bm g,
    \qquad
    A_c := A + \tau n_\Gamma n_\Gamma^T,
    \qquad
    \tau\neq 0 .
    \label{eq:rank-one-completed-A}
\end{align}

\begin{remark}
The parameter $\tau$ is designed to lift the nullspace of the rank-deficient matrix $A$ by introducing new singular values that naturally match the scale of the existing non-zero ones. Rather than computing the smallest non-null singular value by an SVD, we take
\[
\tau=\|A\|_F/\sqrt{n},
\]
where $n$ is the number of columns of $A$. This Frobenius-norm proxy approximates the mean singular-value magnitude and keeps $A_c$ well conditioned without distorting the bulk of its spectrum.
\end{remark}

The rank-one completion in \eqref{eq:rank-one-completed-A} removes the
hydrostatic pressure-jump nullspace.  Indeed, if the velocity vanishes and
the pressure is piecewise constant,
\[
    \bm u\equiv 0,
    \qquad
    p =
    \begin{cases}
    p_{\rm in}, & \text{inside the object},\\
    p_{\rm out}, & \text{outside the object},
    \end{cases}
\]
then the continuous Stokes equation reduces to $\nabla p=\bm q$. In the continuous traction interpretation, this gives a boundary density proportional to the normal, $\bm q\propto -[p]\nu$.
Thus the hydrostatic null mode is a normal mode.  In the exact staggered
discretization, the corresponding algebraic null vector is
\[
    z_h
    =
    \left.
    \nabla_h\chi
    \right|_{\gamma_-},
\]
where \(\chi\) is a cell-centered indicator that is constant away from the
discrete boundary layer.  The sampled vector \(n_\Gamma\) in
\eqref{eq:sampled-normal-vector} is the cut-point representation of the
same mode and is natural when the Dirichlet data are imposed at boundary
intersection points.

\begin{remark}
\label{rem:normal-completion}
For \(K\) disconnected objects, there are \(K\) independent hydrostatic
pressure-jump modes.  Let \(n_\Gamma^{(k)}\) denote the sampled normal vector
supported on the intersection points belonging to object \(k\), and define
\begin{align}
    N_\Gamma
    =
    \begin{bmatrix}
    n_\Gamma^{(1)} & n_\Gamma^{(2)} & \cdots & n_\Gamma^{(K)}
    \end{bmatrix}.
\end{align}
The rank-\(K\) completed matrix is
\begin{align}
    A_c^{(K)}
    =
    A
    +
    \tau
    N_\Gamma N_\Gamma^T,
    \qquad
    \tau\neq 0.
    \label{eq:rank-K-completion}
\end{align}
\end{remark}

Once \(\bm q\) has been determined from \eqref{eq:rank-one-completed-A}, the
velocity and pressure are recovered by the LGF convolutions
\begin{align}
    \bm u = \mathcal S \bm q,
    \qquad
    p = \mathcal P \bm q .
\end{align}

\paragraph{Closure with additional exterior points}
In some cut configurations, the local interpolation stencil contains
additional exterior velocity points beyond \(\gamma_-^u\) and
\(\gamma_-^v\).  Let these additional unknowns be denoted by
\(\eta_u\) and \(\eta_v\).  The boundary interpolation equations become
\begin{subequations}
\label{eq:bc-extra-points}
\begin{align}
    \Phi_+^u u_+ + \Phi_-^u u_- + \Phi_\eta^u \eta_u &= g^u,\\
    \Phi_+^v v_+ + \Phi_-^v v_- + \Phi_\eta^v \eta_v &= g^v.
\end{align}
\end{subequations}
The extra values are eliminated by local extrapolation:
\begin{subequations}
\label{eq:eta-extrapolation}
\begin{align}
    R_1^u u_+ + R_2^u u_- + \eta_u &= 0,
    \qquad
    \eta_u = -R_1^u u_+ - R_2^u u_-,\\
    R_1^v v_+ + R_2^v v_- + \eta_v &= 0,
    \qquad
    \eta_v = -R_1^v v_+ - R_2^v v_- .
\end{align}
\end{subequations}
Substituting \eqref{eq:eta-extrapolation} into
\eqref{eq:bc-extra-points} gives
\begin{subequations}
\label{eq:bc-extra-eliminated}
\begin{align}
    \left(\Phi_+^u-\Phi_\eta^uR_1^u\right)u_+
    +
    \left(\Phi_-^u-\Phi_\eta^uR_2^u\right)u_-
    &=g^u,\\
    \left(\Phi_+^v-\Phi_\eta^vR_1^v\right)v_+
    +
    \left(\Phi_-^v-\Phi_\eta^vR_2^v\right)v_-
    &=g^v.
\end{align}
\end{subequations}
Define the effective interpolation matrices
\begin{align}
    \widetilde\Phi_+
    :=
    \begin{bmatrix}
    \Phi_+^u-\Phi_\eta^uR_1^u & 0\\
    0 & \Phi_+^v-\Phi_\eta^vR_1^v
    \end{bmatrix},
    \qquad
    \widetilde\Phi_-
    :=
    \begin{bmatrix}
    \Phi_-^u-\Phi_\eta^uR_2^u & 0\\
    0 & \Phi_-^v-\Phi_\eta^vR_2^v
    \end{bmatrix}.
\end{align}
Then the boundary condition can be written compactly as
\begin{align}
    \widetilde\Phi_+\bm u_+
    +
    \widetilde\Phi_-\bm u_-
    =
    \bm g.
    \label{eq:bc-extra-compact}
\end{align}
Using \(\bm u_\pm=S_\pm \bm q\), we obtain
\begin{align}
    \widetilde A \bm q=\bm g,
    \qquad
    \widetilde A
    :=
    \widetilde\Phi_+S_+
    +
    \widetilde\Phi_-S_- .
    \label{eq:density-bc-extra}
\end{align}
The corresponding rank-one completed system is
\begin{align}
    \widetilde A_c \bm q=\bm g,
    \qquad
    \widetilde A_c
    :=
    \widetilde A + \tau n_\Gamma n_\Gamma^T,
    \qquad
    \tau\neq 0 .
    \label{eq:rank-one-completed-A-extra}
\end{align}

Equivalently, \(\widetilde A\) has the block form
\begin{align}
\widetilde A
=
\begin{bmatrix}
(\Phi_+^u-\Phi_\eta^uR_1^u)S_{11}^{+}
+
(\Phi_-^u-\Phi_\eta^uR_2^u)S_{11}^{-}
&
(\Phi_+^u-\Phi_\eta^uR_1^u)S_{12}^{+}
+
(\Phi_-^u-\Phi_\eta^uR_2^u)S_{12}^{-}
\\[4pt]
(\Phi_+^v-\Phi_\eta^vR_1^v)S_{21}^{+}
+
(\Phi_-^v-\Phi_\eta^vR_2^v)S_{21}^{-}
&
(\Phi_+^v-\Phi_\eta^vR_1^v)S_{22}^{+}
+
(\Phi_-^v-\Phi_\eta^vR_2^v)S_{22}^{-}
\end{bmatrix}.
\label{eq:block-density-bc-extra}
\end{align}
Although the matrices in \eqref{eq:density-bc-basic} and
\eqref{eq:density-bc-extra} are dense, their matrix-vector products can be
evaluated by LGF convolutions.  Hence the completed systems
\eqref{eq:rank-one-completed-A} and
\eqref{eq:rank-one-completed-A-extra} are suitable for matrix-free GMRES.

\section{Calder\'on preconditioning with the Laplace LGF}
\label{sec:laplace-calderon-preconditioning}

The boundary-density equation obtained above has the form
\begin{align}
    A_c \bm q = \bm g,
    \label{eq:completed-density-system}
\end{align}
where \(A_c\) denotes the rank-completed version of either
\[
    A=\Phi_+S_+ + \Phi_-S_-
\]
or, when additional exterior points are used,
\[
    \widetilde A=\widetilde\Phi_+S_+ + \widetilde\Phi_-S_- .
\]
Although this system is square after the hydrostatic completion, it is dense
and may become ill-conditioned as the boundary is refined.  We therefore use
a preconditioner built from the scalar Laplace lattice Green's function on
the same staggered boundary sets.

Let \(G\) denote the regularized scalar Laplace LGF.  For each velocity
component \(\alpha\in\{u,v\}\), define the scalar Laplace layer matrices
\(V_\alpha\) by
\begin{align}
    (V_\alpha \rho)(m)
    =
    \sum_{n\in\gamma_-^\alpha}
    G(m-n)\rho(n),
    \qquad
    m\in\gamma_-^\alpha.
\end{align}
Thus \(V_\alpha\) maps scalar densities on the exterior layer
\(\gamma_-^\alpha\) to values on the same exterior layer. 
We then define the componentwise Laplace preconditioner
\begin{align}
    V
    :=
    \begin{bmatrix}
    V_u & 0\\
    0 & V_v
    \end{bmatrix}.
    \label{eq:laplace-preconditioner-basic}
\end{align}

The preconditioned system is then solved in left-preconditioned form:
\begin{align}
    V^{-1} A_c  \bm{q} = V^{-1}\bm g.
    \label{eq:preconditioned-system}
\end{align}

\paragraph{Operator-order interpretation}
The effectiveness of the Laplace preconditioner follows from the shared principal singularity of the Laplace and Stokes kernels. In Fourier variables,
\[
\widehat G(\theta)=\frac{h^2}{\lambda(\theta)},
\qquad
\widehat{\mathcal S}(\theta)=-\frac{h^2}{\mu\,\lambda(\theta)}\,\Pi_h(\theta),
\]
where $\Pi_h(\theta)$ is the discrete Leray--Helmholtz projector onto the divergence-free subspace on the staggered MAC grid. Up to the viscosity factor and this bounded matrix-valued projection, the Stokes velocity kernel has the same singular scalar factor as the Laplace kernel. After restriction to a boundary layer, both single-layer operators are of order $-1$, so $V$ captures the leading singular order of $A$ and $V^{-1}A_c$ is expected to behave like an order-zero operator on the constrained density space.

\paragraph{Preconditioning via discrete Calder\'on identities}

The condition numbers of $V_u$ and $V_v$ grow in proportion to the
number of boundary nodes, or $\mathcal O(h^{-1})$ for a resolved smooth
boundary \citep{martinsson2009boundary,xia2026geom}. Although the formal
preconditioner $V^{-1}$ corrects this operator order, forming or inverting the
dense matrix $V$ is expensive. We therefore replace $V^{-1}$ by a
componentwise construction based on discrete Calder\'on identities.

For each velocity component $\alpha\in\{u,v\}$, we first define an unscaled double-difference matrix $W_\alpha$ acting across the discrete boundary interface $\gamma_-^\alpha$. Let $t^-$ and $s^-$ denote a given target node and source node in $\gamma_-^\alpha$, with respective further-exterior neighbor nodes $t^+$ and $s^+$ (one lattice step farther from $\Omega$ across the intersected bar). The matrix elements are constructed via the lattice Green's function as
\begin{equation}
    (W_\alpha)_{t^-s^-} = - \bigl[G(s^+-t^+) - G(s^--t^+) - G(s^+-t^-) + G(s^--t^-)\bigr],\quad s^-,t^-\in \gamma^\alpha_-,
\end{equation}
adapted from \citep{bhamidipati2021forward}.

In an unfitted lattice framework, the mapping between the exterior boundary nodes and the intersected lattice bars is generally not one-to-one; therefore, we must introduce an incidence scaling matrix $J_\alpha$ to balance the topological dimensions. Assuming uniform unit conductance, $J_\alpha$ reduces to a diagonal matrix where each entry $(J_\alpha)_{ii}$ corresponds exactly to the number of exterior lattice connections attached to the boundary node $i \in \gamma_-^\alpha$. The fully formed discrete hyper-singular operator for component $\alpha$ is then given by $H_\alpha = J_\alpha + W_\alpha$.

On each component block, this construction satisfies the discrete
Calder\'on identity \citep{bhamidipati2021forward}
\begin{equation}
    V_\alpha (J_\alpha+W_\alpha) = -K_\alpha (I_\alpha + K_\alpha),
\end{equation}
where $K_\alpha$ is the discrete double-layer operator on
$\gamma_-^\alpha$ and $I_\alpha$ is the identity. In this lattice
normalization, the jump contribution is contained in $K_\alpha$ rather than
written as a separate one-half identity term. For smooth, well-resolved
boundaries, the relevant eigenvalues of $K_\alpha$ cluster near $-1/2$;
hence the polynomial $-K_\alpha(I_\alpha+K_\alpha)$ clusters near $1/4$.

While the product $V_\alpha H_\alpha$ provides the desired spectral clustering, the matrix $H_\alpha$ inherently possesses a one-dimensional null space spanned by the constant vector $\bm{1}$, faithfully mirroring the property of the continuum hyper-singular operator which maps a constant field to zero. If left unregularized, applying the unscaled discrete operator $V_\alpha$ to this constant mode produces an isolated eigenvalue that scales proportionally with the boundary dimension $m_\alpha = |\gamma_-^\alpha|$. This disparate scaling causes the overall condition number of the preconditioned system to double with each mesh refinement.

To control the mesh-dependent growth of the preconditioned spectrum, we deflate the null space for each component via a rank-1 update. Let $\bm{e}_\alpha = \frac{1}{\sqrt{m_\alpha}} \bm{1}$ denote the normalized constant vector. We define the regularized block preconditioner as
\begin{equation}
    H_{\text{reg},\alpha} = H_\alpha + c_\alpha \bm{e}_\alpha \bm{e}_\alpha^T,
\end{equation}
where we introduce $c_\alpha$ as the shift parameter. To minimize the overall spectral spread, $c_\alpha$ is chosen such that the isolated constant mode is mapped directly into the center of the existing Calder\'on eigenvalue cluster at $0.25$. By measuring the spectral scaling of the constant mode under the single-layer operator as $\lambda_{\alpha} = \bm{e}_\alpha^T V_\alpha \bm{e}_\alpha$, the mathematically optimal scaling parameter is determined to be
\begin{equation}
    c_\alpha = \frac{0.25}{\lambda_\alpha}.
\end{equation}

Finally, the global componentwise preconditioner is assembled as a block-diagonal matrix,
\begin{equation}
    H_{\text{reg}} := 
    \begin{bmatrix}
    H_{\text{reg},u} & 0\\
    0 & H_{\text{reg},v}
    \end{bmatrix},
\end{equation}
which replaces $V^{-1}$ in \eqref{eq:preconditioned-system} to form the robust left-preconditioned system 
\begin{align}
H_{\text{reg}} A_c \bm{q} = H_{\text{reg}}\bm{g},
\end{align}
with no inverse of dense matrices.

\begin{remark}[Generalization to multiple disjoint objects]


When we have $M$ objects, let $\bm{e}^{(k)}_\alpha$ denote the normalized indicator vector for the $k$-th object on the $\alpha$-component grid, defined such that its $i$-th entry is $1/\sqrt{m_{\alpha,k}}$ if boundary node $i$ belongs to object $k$, and $0$ otherwise. Because the object boundaries are spatially disjoint, these indicator vectors are strictly orthogonal. The fully regularized block preconditioner is constructed by independently deflating each topological component:
\begin{equation}
    H_{\text{reg},\alpha} = H_\alpha + \sum_{k=1}^M c_{\alpha,k} \bm{e}^{(k)}_\alpha (\bm{e}^{(k)}_\alpha)^T.
\end{equation}
To preserve the tight Calder\'on spectral clustering, the optimal shift parameter is calculated independently for each object's constant mode via
\begin{equation}
    c_{\alpha,k} = \frac{0.25}{(\bm{e}^{(k)}_\alpha)^T V_\alpha \bm{e}^{(k)}_\alpha}.
\end{equation}
This rank-$M$ completion maps each resolved componentwise constant
mode to the target value $0.25$. It therefore preserves the intended spectral
clustering when the lattice resolves the gaps between distinct objects.
\end{remark}

\section{Inhomogeneous body forcing}
\label{sec:inhomogeneous_forcing}

While the preceding construction addressed the homogeneous Stokes equations, extending this framework to include nonzero body forces is straightforward and leaves the boundary algebraic operator unchanged. Consider the inhomogeneous boundary value problem:
\begin{equation}
-\mu \Delta \bm u + \nabla p = \bm f,
\qquad
\nabla\cdot \bm u = 0
\quad \text{in } \Omega,
\qquad
\bm u = \bm g
\quad \text{on } \Gamma .
\label{eq:stokes_continuous}
\end{equation}
We decompose the numerical solution into a particular volume potential and a homogeneous boundary correction:
\begin{equation}
\bm u_h = \bm u_h^{f} + \mathcal S_h \bm q,
\qquad
p_h = p_h^{f}+ \mathcal P_h \bm q .
\label{eq:stokes_decomposition}
\end{equation}
Here $(\bm u_h^{f},p_h^{f})$ is a particular volume potential,
whereas $(\mathcal S_h\bm q,\mathcal P_h\bm q)$ is a homogeneous boundary
correction with unknown layer density $\bm q$.

Let $\bm f_h=(f_{1,h},f_{2,h})$ denote the body force sampled on
the two staggered velocity grids. The stored LGFs invert the
discrete Stokes operator
\begin{equation}
-\mu \Delta_h \bm u_h + \nabla_h p_h = \bm F_h,
\qquad
\nabla_h\cdot \bm u_h = 0,
\label{eq:stokes_discrete}
\end{equation}
where $\bm F_h=(F_{1,h},F_{2,h})$ denotes the discrete right-hand side associated with the chosen MAC scaling of the momentum equations (cf.~\eqref{eq:mac-u}--\eqref{eq:mac-v}).

The particular solution is computed via discrete convolution:
\begin{equation}
\begin{aligned}
u_h^{f} &= S_{11,h} * F_{1,h} + S_{12,h} * F_{2,h}, \\
v_h^{f} &= S_{21,h} * F_{1,h} + S_{22,h} * F_{2,h}, \\
p_h^{f} &= P_{1,h} * F_{1,h} + P_{2,h} * F_{2,h},
\end{aligned}
\label{eq:discrete_convolution}
\end{equation}
where $S_{ij,h}$ are the discrete Stokes velocity kernels and $P_{i,h}$ are the corresponding pressure kernels.

Enforcing the prescribed Dirichlet data requires subtracting the
trace of the particular solution from the boundary right-hand side:
\begin{equation}
A_c \bm q = \bm g_h - \Phi_+ \bm u_+^{f}-\Phi_- \bm u_-^{f},
\label{eq:bae_inhomogeneous}
\end{equation}
where $u_\pm^{f}$ denotes the trace of the particular solution at their corresponding boundary nodes.
Thus $A_c$ is unchanged from the homogeneous problem; only its
right-hand side changes. Formally,
\begin{equation}
\bm q = A_c^{-1}\bigl(\bm g_h - \Phi_+ \bm u_+^{f}-\Phi_- \bm u_-^{f}\bigr).
\label{eq:solve_q}
\end{equation}
In practice, the action of $A_c^{-1}$ in \eqref{eq:solve_q} is
computed by preconditioned GMRES rather than by forming an inverse. The final
velocity and pressure follow from \eqref{eq:stokes_decomposition}; as usual,
the pressure is determined only up to an additive constant.

This decomposition also provides a natural pathway for solving time-dependent incompressible flows. In a semi-implicit Navier--Stokes discretization, the nonlinear and temporal terms can be treated as known volume sources at each time step, allowing the boundary algebraic equation to enforce the no-slip condition entirely through the homogeneous correction.

\section{FFT acceleration of LGF convolutions}
\label{sec:fft_acceleration}

The discrete Stokes layer potentials and volume potentials are translation-invariant convolutions evaluated on the underlying Cartesian lattice. Direct evaluation of these potentials is computationally expensive; for instance, a direct volume convolution on an $N_1\times N_2$ grid requires $O(N_1^2N_2^2)$ operations. Boundary-to-volume reconstruction is similarly costly if evaluated by dense summation at all target grid points. 

These computational costs are significantly reduced by embedding the linear convolution into a circulant convolution and applying the Fast Fourier Transform (FFT). For a generic kernel $K_h$ and source $F_h$, this is expressed as:
\begin{equation}
K_h * F_h = \operatorname{crop}
\left[
\mathcal F^{-1}
\left(
\mathcal F(K_h^{\rm pad}) \cdot \mathcal F(F_h^{\rm pad})
\right)
\right],
\label{eq:fft_general}
\end{equation}
where $\mathcal F$ denotes the two-dimensional FFT, pointwise multiplication is applied in frequency space, and the crop operation extracts the physical target region. This approach reduces the operational cost of a volume convolution to:
\begin{equation}
O(M_1M_2\log(M_1M_2)),
\label{eq:fft_complexity}
\end{equation}
where $M_1\times M_2$ are the dimensions of the padded arrays (typically $M_i\approx 2N_i$ for an $N_1\times N_2$ target grid).

The same procedure applies componentwise to the Stokes velocity and pressure kernels. For a staggered-grid force $\bm F_h=(F_{1,h},F_{2,h})$, the volume potential is efficiently evaluated by:
\begin{equation}
\begin{aligned}
u_h^{f} &= \operatorname{crop} \left[ \mathcal F^{-1} \left( \widehat S_{11,h}\widehat F_{1,h} + \widehat S_{12,h}\widehat F_{2,h} \right) \right], \\
v_h^{f} &= \operatorname{crop} \left[ \mathcal F^{-1} \left( \widehat S_{21,h}\widehat F_{1,h} + \widehat S_{22,h}\widehat F_{2,h} \right) \right], \\
p_h^{f} &= \operatorname{crop} \left[ \mathcal F^{-1} \left( \widehat P_{1,h}\widehat F_{1,h} + \widehat P_{2,h}\widehat F_{2,h} \right) \right],
\end{aligned}
\label{eq:fft_volume_potential}
\end{equation}
where hats denote the FFTs of the padded arrays.

Boundary layer reconstruction is handled using an identical framework. The boundary density is first distributed onto the corresponding staggered grid locations to form grid-supported source arrays, $Q_{1,h}$ and $Q_{2,h}$. The homogeneous velocity and pressure corrections are then computed via:
\begin{equation}
\begin{aligned}
\mathcal S_{1,h}\bm q &= \operatorname{crop} \left[ \mathcal F^{-1} \left( \widehat S_{11,h}\widehat Q_{1,h} + \widehat S_{12,h}\widehat Q_{2,h} \right) \right], \\
\mathcal S_{2,h}\bm q &= \operatorname{crop} \left[ \mathcal F^{-1} \left( \widehat S_{21,h}\widehat Q_{1,h} + \widehat S_{22,h}\widehat Q_{2,h} \right) \right], \\
\mathcal P_h\bm q &= \operatorname{crop} \left[ \mathcal F^{-1} \left( \widehat P_{1,h}\widehat Q_{1,h} + \widehat P_{2,h}\widehat Q_{2,h} \right) \right].
\end{aligned}
\label{eq:fft_boundary_reconstruction}
\end{equation}


For repeated solves on a fixed grid, the FFTs of the padded kernels
\begin{equation}
\widehat S_{11,h},\quad
\widehat S_{12,h},\quad
\widehat S_{21,h},\quad
\widehat S_{22,h},\quad
\widehat P_{1,h},\quad
\widehat P_{2,h},
\label{eq:precomputed_kernels}
\end{equation}
are precomputed once and cached. Only the FFTs of the transient source arrays or boundary density arrays need to be computed during each solve. This precomputation is highly advantageous for moving-boundary or time-dependent problems, as the background grid and lattice Green's functions remain static while only the geometry-dependent interpolation and boundary algebraic equation require updates.

\section{Algorithm}
\label{sec:end-to-end-algorithm}

Algorithm~\ref{alg:bae-stokes} collects the preceding construction
into a single computational procedure. The kernel tables depend only on the
background lattice and can be reused, whereas the boundary layers,
interpolation matrices, rank completions, and preconditioner depend on the
current geometry.

\begin{algorithm}[H]
\algFontSize
\caption{Unfitted BAE solver for the steady Stokes equations}
\label{alg:bae-stokes}
\begin{algorithmic}[1]
\Require Physical boundary $\Gamma$, MAC-grid spacing $h$, viscosity $\mu$,
body force $\bm f$, Dirichlet data $\bm g$, and GMRES tolerance $\varepsilon$
\Ensure Staggered velocity $\bm u_h$ and cell-centered pressure $p_h$
\State Classify the staggered velocity nodes into $M_\pm^u$ and $M_\pm^v$;
construct the boundary layers $\gamma_\pm^u$ and $\gamma_\pm^v$.
\State Locate the grid-line intersections with $\Gamma$ and assemble
$\Phi_\pm$; if extra exterior nodes enter a cut stencil, eliminate them by
local extrapolation to obtain $\widetilde\Phi_\pm$.
\State Tabulate the regularized Laplace LGF $G$ and the biharmonic LGF $H$;
form the Stokes velocity kernels $S_{ij,h}$ and pressure kernels $P_{i,h}$.
\State Form the boundary operator $A$ from the interpolation matrices and
the interior/exterior traces of the Stokes layer potential.
\State Add one sampled-normal update for each independent hydrostatic mode,
$A_c=A+\tau N_\Gamma N_\Gamma^T$.
\State Build the componentwise Calder\'on preconditioner $H_{\mathrm{reg}}$
from $V_\alpha$, $J_\alpha$, and $W_\alpha$, including one constant-mode
update per resolved object.
\If{$\bm f\not\equiv\bm 0$}
    \State
    Form the discrete right-hand side $\bm F_h$ and evaluate the
    particular fields $(\bm u_h^f,p_h^f)$ by LGF convolution.
\Else
    \State Set $(\bm u_h^f,p_h^f)=(\bm 0,0)$.
\EndIf
\State Form the effective boundary data
$\bm b_h=\bm g_h-\Phi_+ \bm u_+^{f}-\Phi_- \bm u_-^{f}$.
\State Solve
$H_{\mathrm{reg}}A_c\bm q=H_{\mathrm{reg}}\bm b_h$ by GMRES to tolerance
$\varepsilon$; apply the LGF convolutions directly or by padded FFTs.
\State Reconstruct
$\bm u_h=\bm u_h^f+\mathcal S_h\bm q$ and
$p_h=p_h^f+\mathcal P_h\bm q$.
\State Fix the additive pressure constant according to the chosen gauge.
\end{algorithmic}
\end{algorithm}

\section{Numerical experiments}
\label{sec:numerical-experiments}

We assess four features of the method: accuracy under grid
refinement, performance of the Calder\'on preconditioner, preservation of the
discrete divergence constraint, and robustness with respect to geometry and
domain topology. The tests progress from smooth interior domains to narrow
gaps, body forcing, corner eddies, and exterior domains with multiple
obstacles.

Unless stated otherwise, $\mu=1$
(without loss of generality for steady Stokes flow, which is linear in $\mu$),
and consecutive values of $N$
approximately double the linear grid resolution in the computational domain $[-1.15,1.15]\times[-1.15,1.15]$.
Here $N$ denotes the number of MAC cells in each coordinate direction of that box.
We report componentwise
maximum-norm errors over all fluid-grid nodes:
\begin{equation}
  E_u=\max_{j,k}|u_{j,k}-u_{j,k}^{\rm exact}|,\qquad
  E_v=\max_{j,k}|v_{j,k}-v_{j,k}^{\rm exact}|,\qquad
  E_p=\max_{j,k}|p_{j,k}-p_{j,k}^{\rm exact}|.
  \label{eq:numerical-error-norms}
\end{equation}
Because pressure is defined only up to a constant, the numerical and
exact pressures are aligned to the same gauge before $E_p$ is evaluated. All
GMRES solves use the Calder\'on preconditioner and a stopping tolerance of
$10^{-10}$ unless stated otherwise. Solver tables list the
condition number of the rank-completed operator $\kappa(A_c)$ and of the preconditioned operator $\kappa(H_{\rm reg}A_c)$,
the GMRES iteration count, and the maximum discrete
divergence $\|\nabla_h\cdot\bm u\|_\infty$.

Several convergence tests use the following force-doublet solution.
Place a Stokeslet with force $\bm f$ at $\bm x_0$ and an opposing
Stokeslet with force $-\bm f$ at $-\bm x_0$. With
$\bm r^+=\bm x-\bm x_0$, $\bm r^-=\bm x+\bm x_0$, and
$r^\pm=|\bm r^\pm|$, the resulting velocity and pressure are
\begin{equation}\label{eq:doublet}
\begin{aligned}
  u_{i}(\bm{x})
  &= \frac{1}{4\pi\mu}
    \left[ \left(-\delta_{ij}\ln r^{+} + \frac{r_{i}^{+}\,r_{j}^{+}}{(r^{+})^{2}}\right)
         - \left(-\delta_{ij}\ln r^{-} + \frac{r_{i}^{-}\,r_{j}^{-}}{(r^{-})^{2}}\right) \right] f_{j}, \\
  p(\bm{x})
  &= \frac{1}{2\pi} \left[ \frac{r_{j}^{+}}{(r^{+})^{2}} - \frac{r_{j}^{-}}{(r^{-})^{2}} \right] f_{j}.
\end{aligned}
\end{equation}
When both singularities lie outside the fluid region, these fields
are smooth exact solutions in the computational domain and supply consistent
Dirichlet data for a convergence study.
In~\eqref{eq:doublet}, repeated indices are summed.

\subsection{Interior flows}

\subsubsection{Ellipse aspect-ratio study}
We first consider the family of ellipses
$x^2+\alpha^2 y^2=1$ with $\alpha\in\{1,2,4,8\}$. The two Stokeslets are
located at $(2,2)$ and $(-2,-2)$, with forces $(10,1)^T$ and
$(-10,-1)^T$, respectively. These sources remain outside every ellipse, and
the exact doublet velocity supplies the boundary data. Varying $\alpha$
tests whether the cut-cell closure remains accurate as the geometry becomes
increasingly slender.

\begin{figure}[htbp]
\centering
    \begin{subfigure}{0.3\textwidth}
        \centering
        \includegraphics[width=\textwidth, trim = 2cm 7cm 2cm 6.5cm]{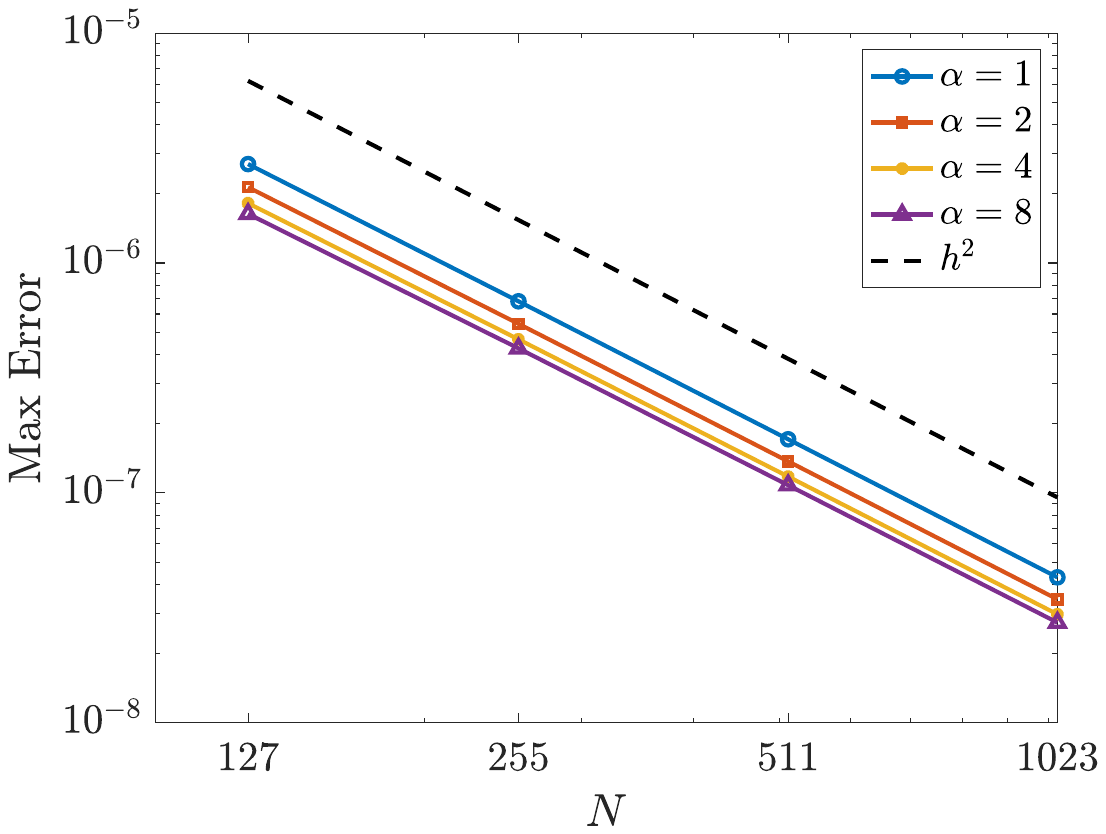}
        \caption{$u$}\label{fig:u_conv}
    \end{subfigure}
    ~
    \begin{subfigure}{0.3\textwidth}
        \centering
        \includegraphics[width=\textwidth, trim = 2cm 7cm 2cm 6.5cm]{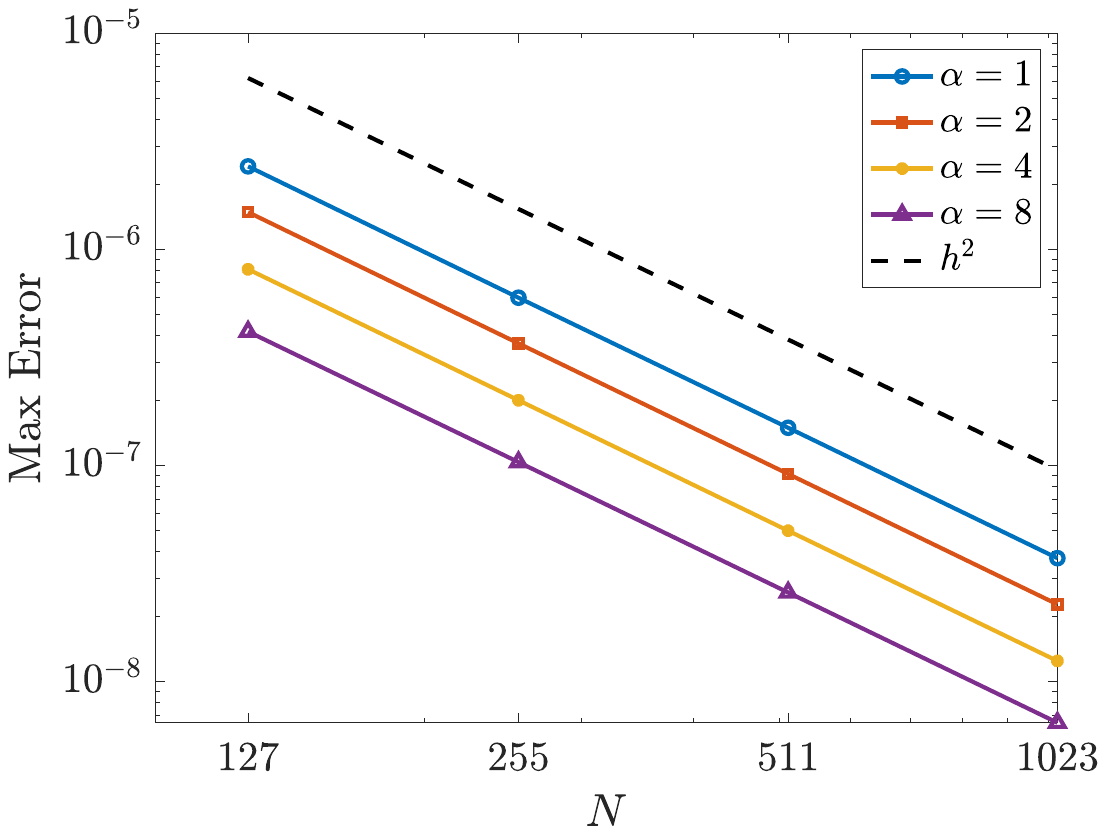}
        \caption{$v$}\label{fig:v_conv}
    \end{subfigure}
    ~
    \begin{subfigure}{0.3\textwidth}
        \centering
        \includegraphics[width=\textwidth, trim = 2cm 7cm 2cm 6.5cm]{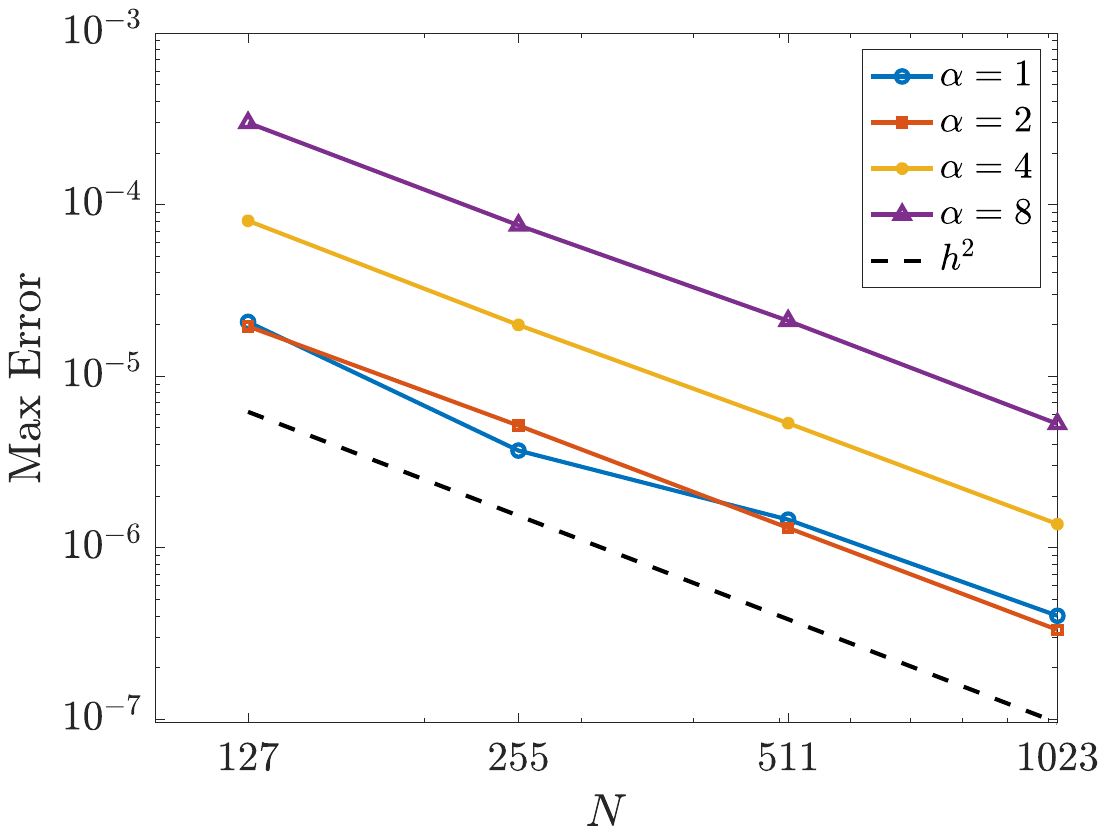}
        \caption{$p$}\label{fig:p_conv}
    \end{subfigure}
    \caption{Maximum-norm convergence for ellipses with different aspect ratios $\alpha$. The dashed line has slope two.}\label{fig:convergence}
\end{figure}

\begin{figure}[htbp]
\centering
    \begin{subfigure}{0.23\textwidth}
        \centering
        \includegraphics[width=\textwidth, trim = 2cm 7cm 2cm 6.5cm]{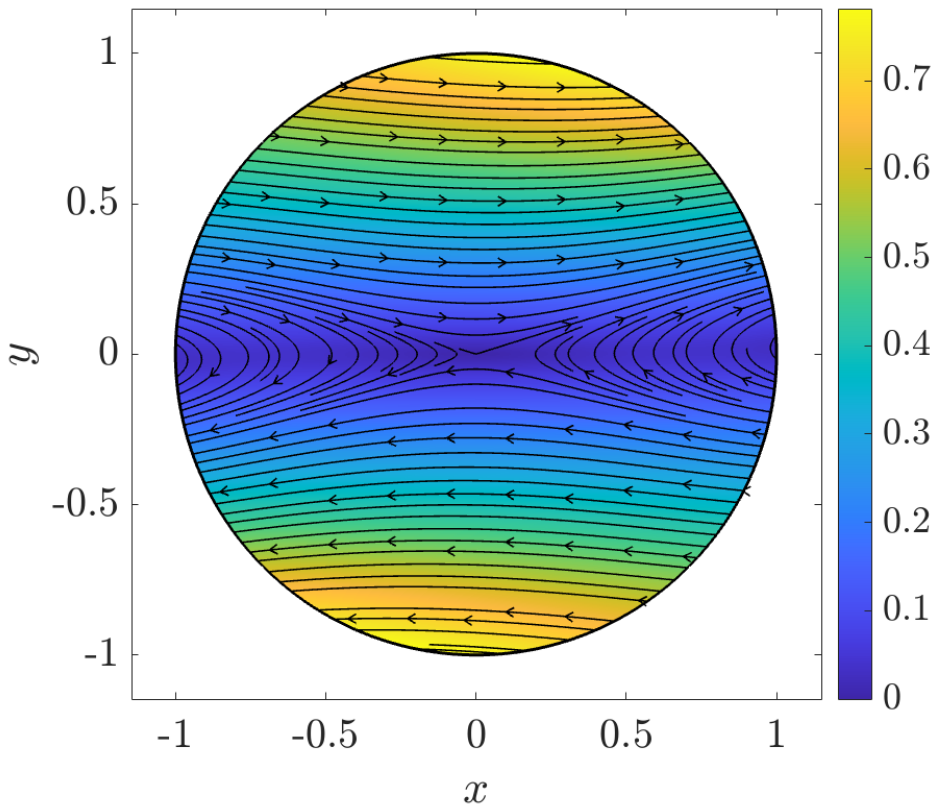}
        \caption{$\alpha=1$}\label{fig:speed-ar1}
    \end{subfigure}
    ~
    \begin{subfigure}{0.23\textwidth}
        \centering
        \includegraphics[width=\textwidth, trim = 2cm 7cm 2cm 6.5cm]{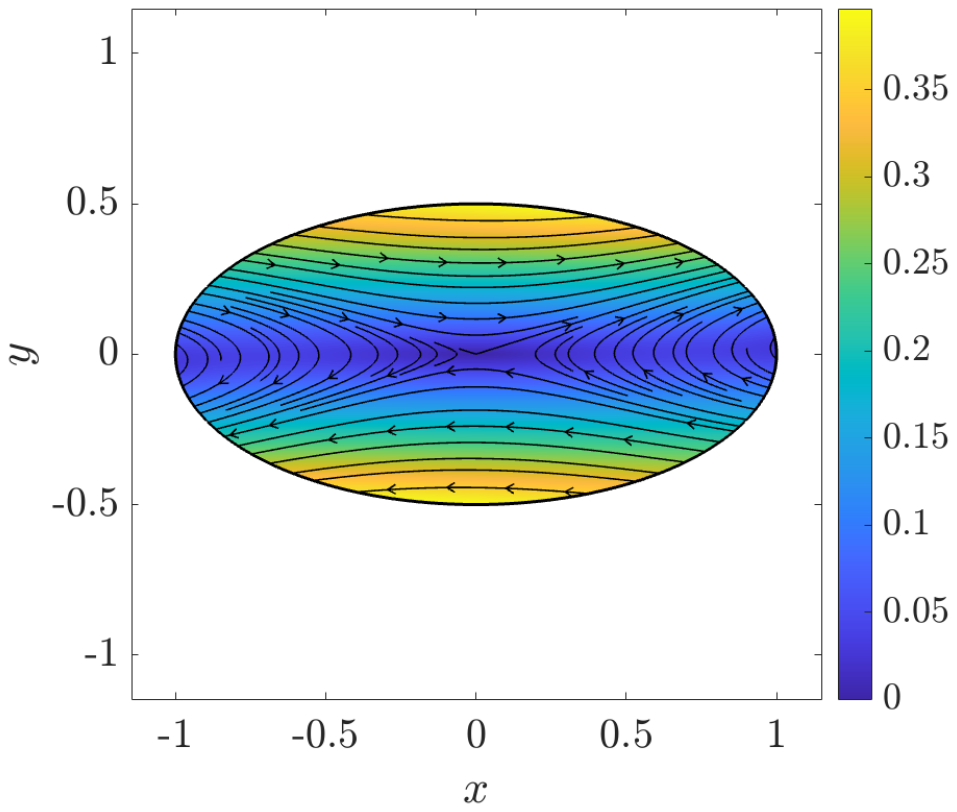}
        \caption{$\alpha=2$}\label{fig:speed-ar2}
    \end{subfigure}
    ~
    \begin{subfigure}{0.23\textwidth}
        \centering
        \includegraphics[width=\textwidth, trim = 2cm 7cm 2cm 6.5cm]{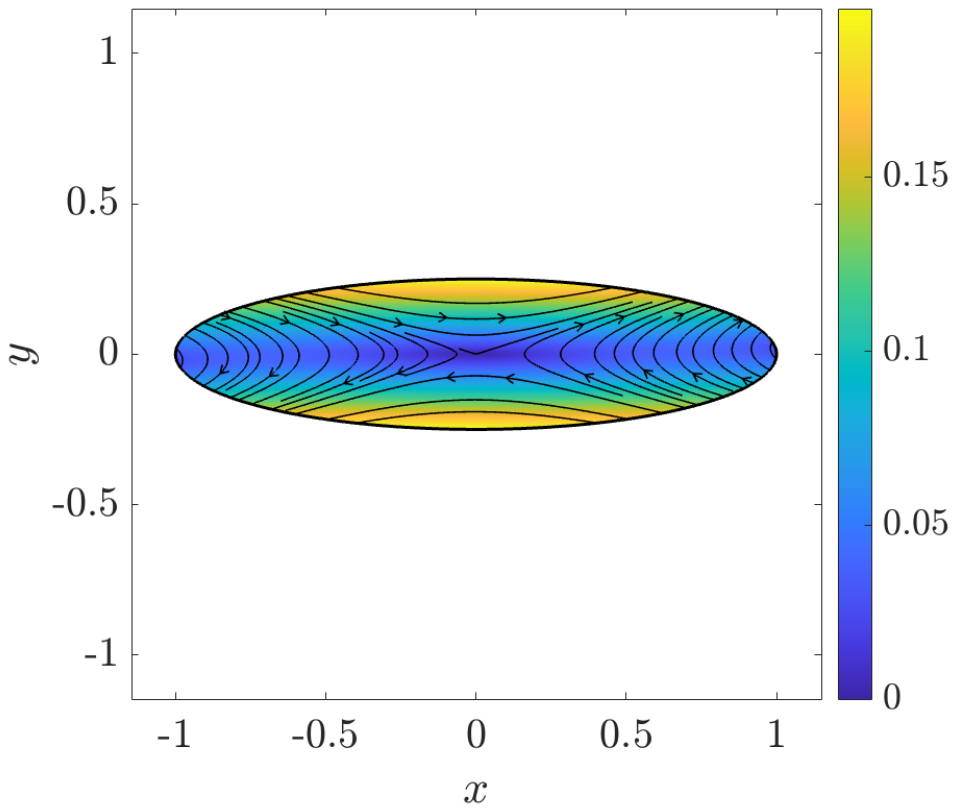}
        \caption{$\alpha=4$}\label{fig:speed-ar4}
    \end{subfigure}
    ~
    \begin{subfigure}{0.23\textwidth}
        \centering
        \includegraphics[width=\textwidth, trim = 2cm 7cm 2cm 6.5cm]{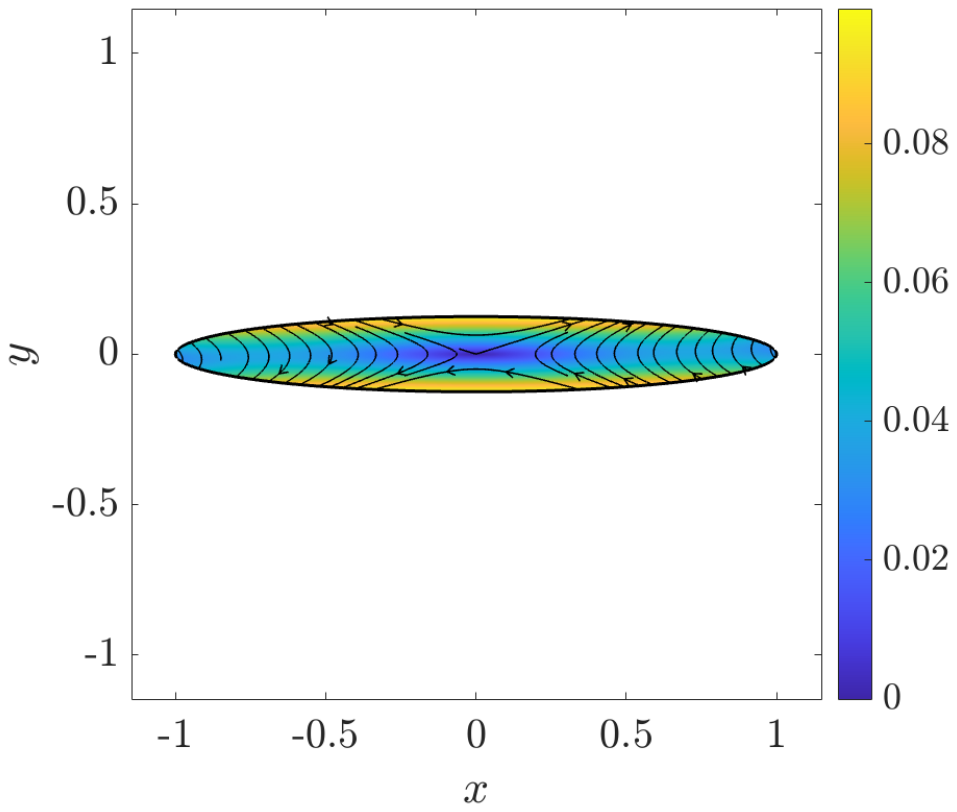}
        \caption{$\alpha=8$}\label{fig:speed-ar8}
    \end{subfigure}
    ~
    \begin{subfigure}{0.23\textwidth}
        \centering
        \includegraphics[width=\textwidth, trim = 2cm 7cm 2cm 6.5cm]{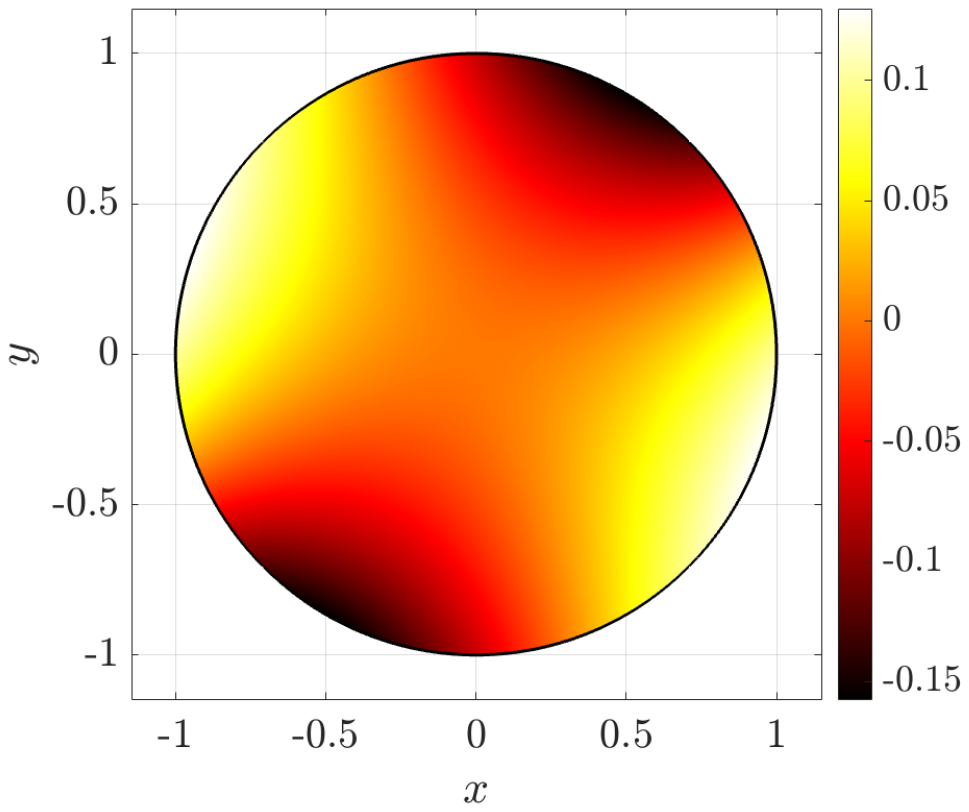}
        \caption{$\alpha=1$}\label{fig:pressure-ar1}
    \end{subfigure}
    ~
    \begin{subfigure}{0.23\textwidth}
        \centering
        \includegraphics[width=\textwidth, trim = 2cm 7cm 2cm 6.5cm]{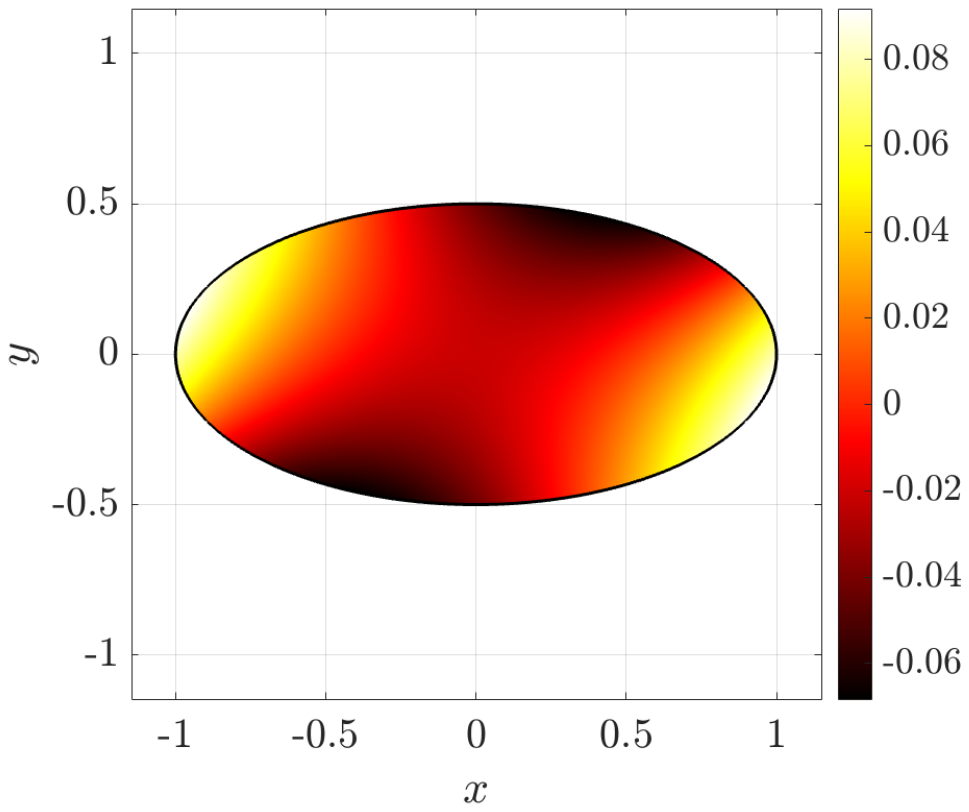}
        \caption{$\alpha=2$}\label{fig:pressure-ar2}
    \end{subfigure}
    ~
    \begin{subfigure}{0.23\textwidth}
        \centering
        \includegraphics[width=\textwidth, trim = 2cm 7cm 2cm 6.5cm]{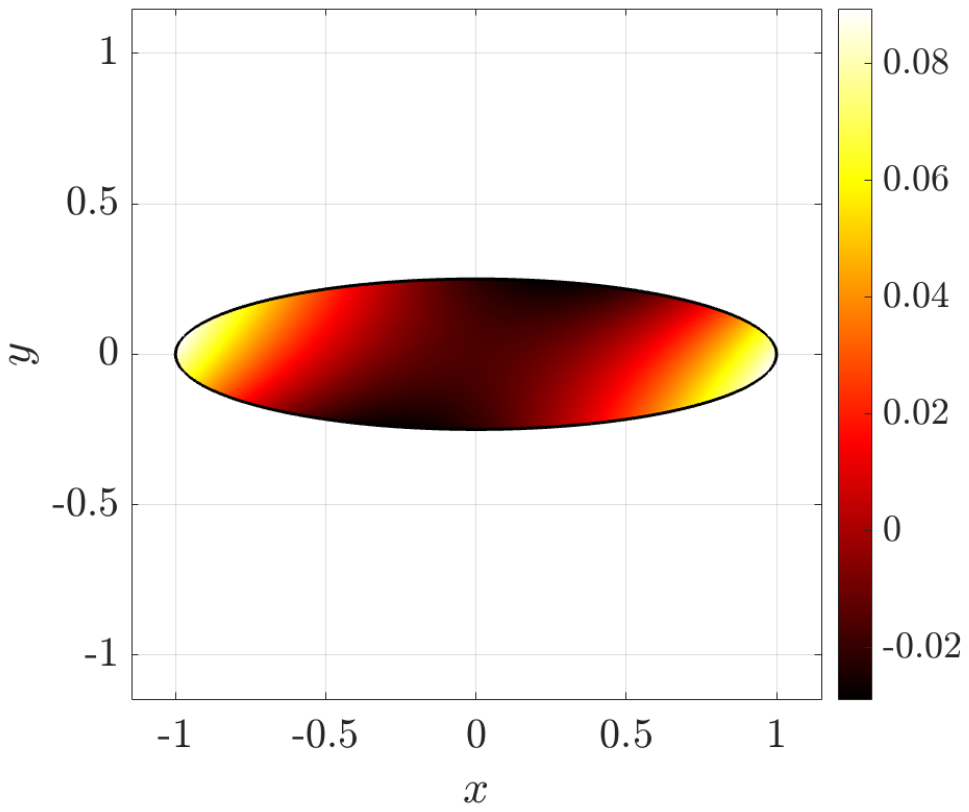}
        \caption{$\alpha=4$}\label{fig:pressure-ar4}
    \end{subfigure}
    ~
    \begin{subfigure}{0.23\textwidth}
        \centering
        \includegraphics[width=\textwidth, trim = 2cm 7cm 2cm 6.5cm]{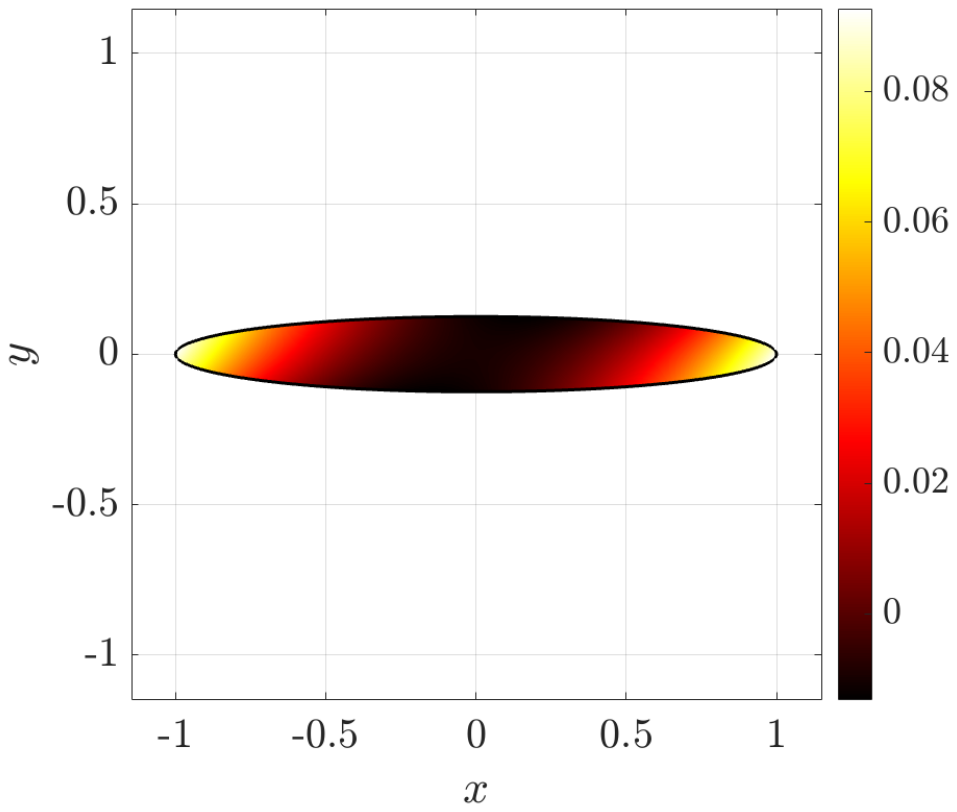}
        \caption{$\alpha=8$}\label{fig:pressure-ar8}
    \end{subfigure}
    \caption{Velocity magnitude with streamlines (top) and pressure (bottom) for ellipse aspect ratios $\alpha=1,2,4,8$ at $N=1023$.}\label{fig:vel-pressure}
\end{figure}

Figure~\ref{fig:convergence} shows nearly parallel second-order
decay for both velocity components across all four aspect ratios. Pressure
also approaches second-order convergence, although its absolute error grows
as the ellipse becomes more slender. The fields in
Figure~\ref{fig:vel-pressure} remain smooth up to the unfitted boundary and
show no visible grid-aligned artifacts.

\begin{table}[htbp]
    \centering
    \small
    \begin{tabular}{@{} S[table-format=4.0] c c @{}}
        \toprule
        {$N$} & {$\alpha=1$} & {$\alpha=2$} \\
        \midrule
         127 & \num{6.0599e4} / \num{1.2440e2} (71) & \num{4.2317e4} / \num{9.8453e1} (72) \\
         255 & \num{9.0113e4} / \num{7.7860e1} (59) & \num{1.2041e5} / \num{1.3301e2} (87) \\
         511 & \num{3.3379e5} / \num{1.4656e2} (90) & \num{2.3935e5} / \num{1.2087e2} (88) \\
        1023 & \num{1.1605e7} / \num{2.2004e3} (93) & \num{8.8150e5} / \num{1.5970e2} (106) \\
        \midrule
        {$N$} & {$\alpha=4$} & {$\alpha=8$} \\
        \midrule
         127 & \num{3.7070e4} / \num{1.0368e2} (74)  & \num{3.6958e4} / \num{2.1202e2} (79) \\
         255 & \num{1.3311e5} / \num{1.4817e2} (87)  & \num{9.6353e4} / \num{1.7869e2} (91) \\
         511 & \num{2.8017e5} / \num{1.5818e2} (101) & \num{2.4961e5} / \num{1.7878e2} (103) \\
        1023 & \num{5.1673e5} / \num{1.3307e2} (100) & \num{1.0583e6} / \num{2.5829e2} (123) \\
        \bottomrule
    \end{tabular}
    \caption{Condition numbers before and after preconditioning, reported as $\kappa(A_c)/\kappa(H_{\rm reg}A_c)$, with GMRES iteration counts in parentheses.}
    \label{tab:ellipse_metrics}
\end{table}

Table~\ref{tab:ellipse_metrics} shows that preconditioning reduces
the condition number by two to four orders of magnitude. Across all sixteen
geometry--resolution combinations, GMRES requires between 59 and 123
iterations. The iteration count grows moderately with aspect ratio, but the
preconditioner remains effective even for $\alpha=8$.

\begin{table}[htbp]
    \centering
    \small 
    \begin{tabular}{@{} S[table-format=4.0] *{4}{S[table-format=1.4e-2]} @{}}
        \toprule
        {$N$} & {$\nabla_h\cdot\bm{u}\ (\alpha=1)$} & {$\nabla_h\cdot\bm{u}\ (\alpha=2)$} & {$\nabla_h\cdot\bm{u}\ (\alpha=4)$} & {$\nabla_h\cdot\bm{u}\ (\alpha=8)$} \\
        \midrule
         127 & 1.1792e-11 & 2.1749e-12 & 1.0691e-12 & 1.0618e-12 \\
         255 & 5.4900e-11 & 9.4904e-12 & 5.6735e-12 & 4.9016e-12 \\
         511 & 4.5181e-10 & 5.1704e-11 & 2.7782e-11 & 2.2669e-11 \\
        1023 & 2.5178e-09 & 3.8415e-10 & 2.1967e-10 & 1.3543e-10 \\
        \bottomrule
    \end{tabular}
    \caption{Maximum discrete divergence for ellipses of different aspect ratios and grid resolutions.}
    \label{tab:divergence}
\end{table}

The divergence values in Table~\ref{tab:divergence} remain between
$10^{-12}$ and $2.5\times10^{-9}$. Their mild increase on finer grids is
consistent with accumulated floating-point and iterative-solver error; it
does not indicate a loss of the algebraic MAC divergence constraint.

\subsubsection{Taylor--Couette flow in an annulus}
\label{sec:taylor_couette}

We next consider steady Taylor--Couette flow between two concentric circles, whose circular boundaries cut the Cartesian grid arbitrarily. Unlike the preceding manufactured solution, this
benchmark has a simple radial profile and therefore directly tests cut-cell
interpolation and rotational symmetry on the Cartesian MAC grid.

The inner and outer radii are $R_i=0.3$ and $R_o=1$. The inner
cylinder rotates with angular velocity $\Omega_i=1$, while the outer cylinder
is stationary, $\Omega_o=0$. In the Stokes limit the pressure is constant,
and the exact azimuthal velocity is
\begin{equation}
    u_\theta(r) = A r + \frac{B}{r}, \quad A = \frac{\Omega_o R_o^2 - \Omega_i R_i^2}{R_o^2 - R_i^2}, \quad B = \frac{(\Omega_i - \Omega_o) R_i^2 R_o^2}{R_o^2 - R_i^2},
\end{equation}
where $r=\sqrt{x^2+y^2}$. The Cartesian components are
\begin{align*}
u(x,y) &= -u_\theta(r) \frac{y}{r},\\
v(x,y) &= u_\theta(r) \frac{x}{r}.
\end{align*} 
The annular boundary is represented by the level-set function
$\phi(x,y)=\max(r-R_o,R_i-r)$.

\begin{figure}[htbp]
\centering
    \begin{subfigure}{0.3\textwidth}
        \centering
        \includegraphics[width=\textwidth, trim = 2cm 7cm 2cm 6.5cm]{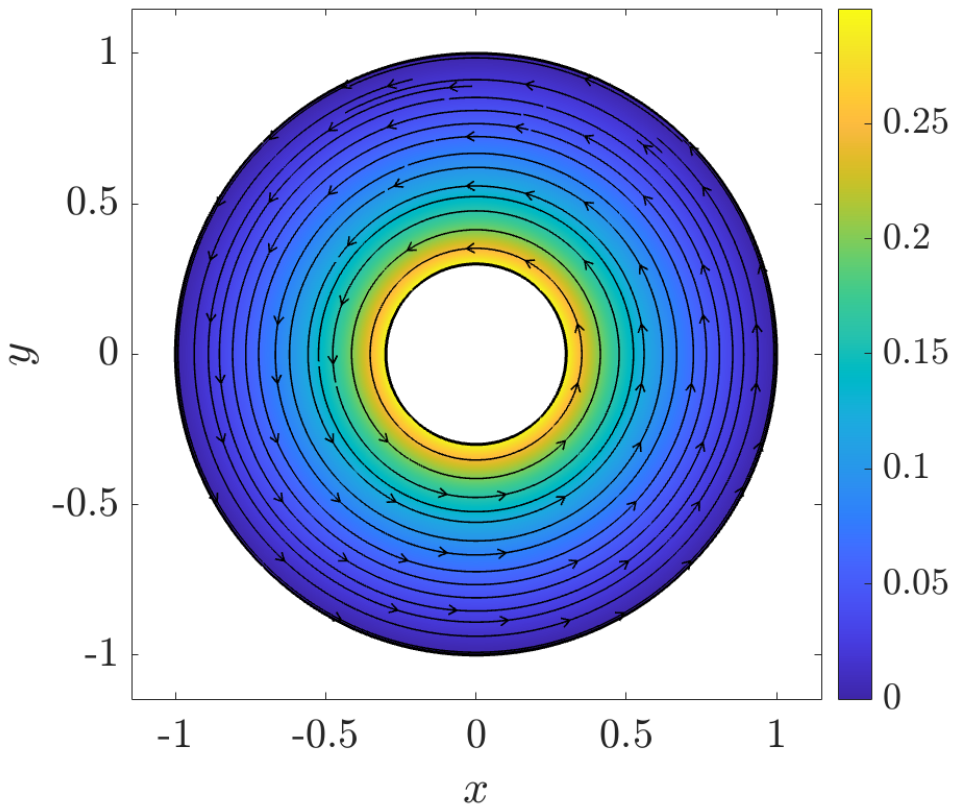}
        \caption{Velocity}\label{fig:couette_speed}
    \end{subfigure}
    ~
    \begin{subfigure}{0.3\textwidth}
        \centering
        \includegraphics[width=\textwidth, trim = 2cm 7cm 2cm 6.5cm]{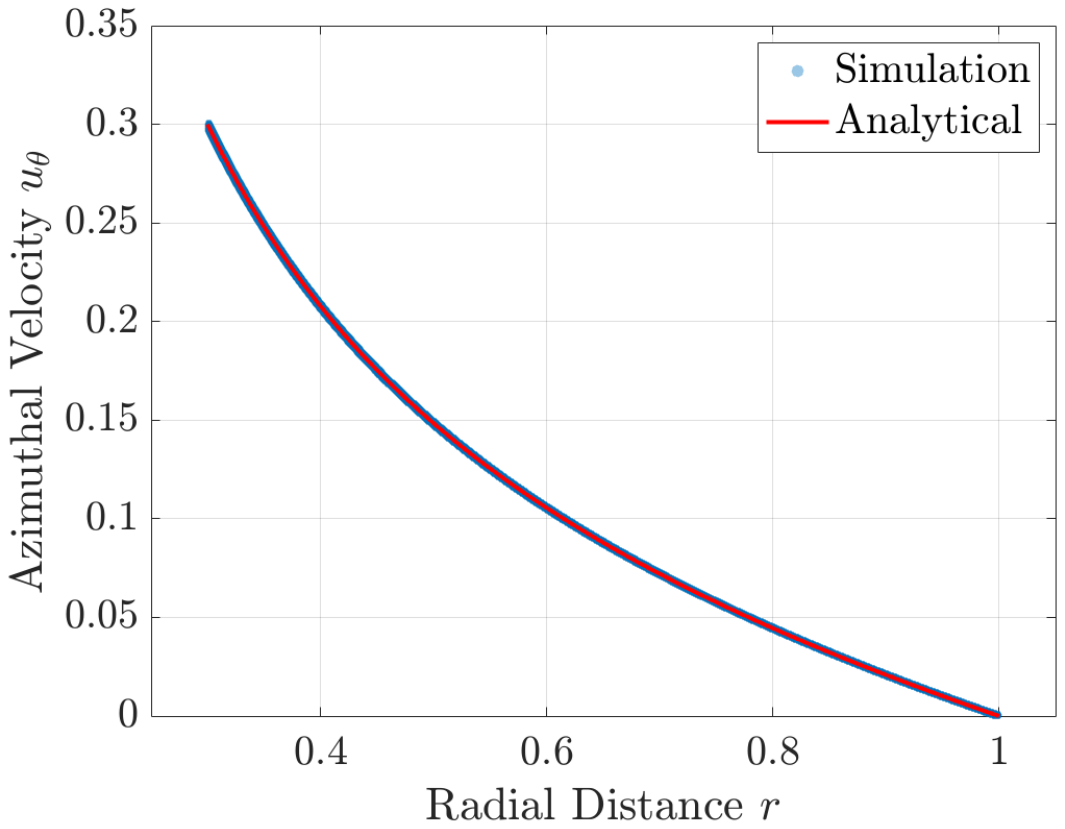}
        \caption{Azimuthal velocity}\label{fig:couette_azimuthal}
    \end{subfigure}
    ~
    \begin{subfigure}{0.3\textwidth}
        \centering
        \includegraphics[width=\textwidth, trim = 2cm 7cm 2cm 6.5cm]{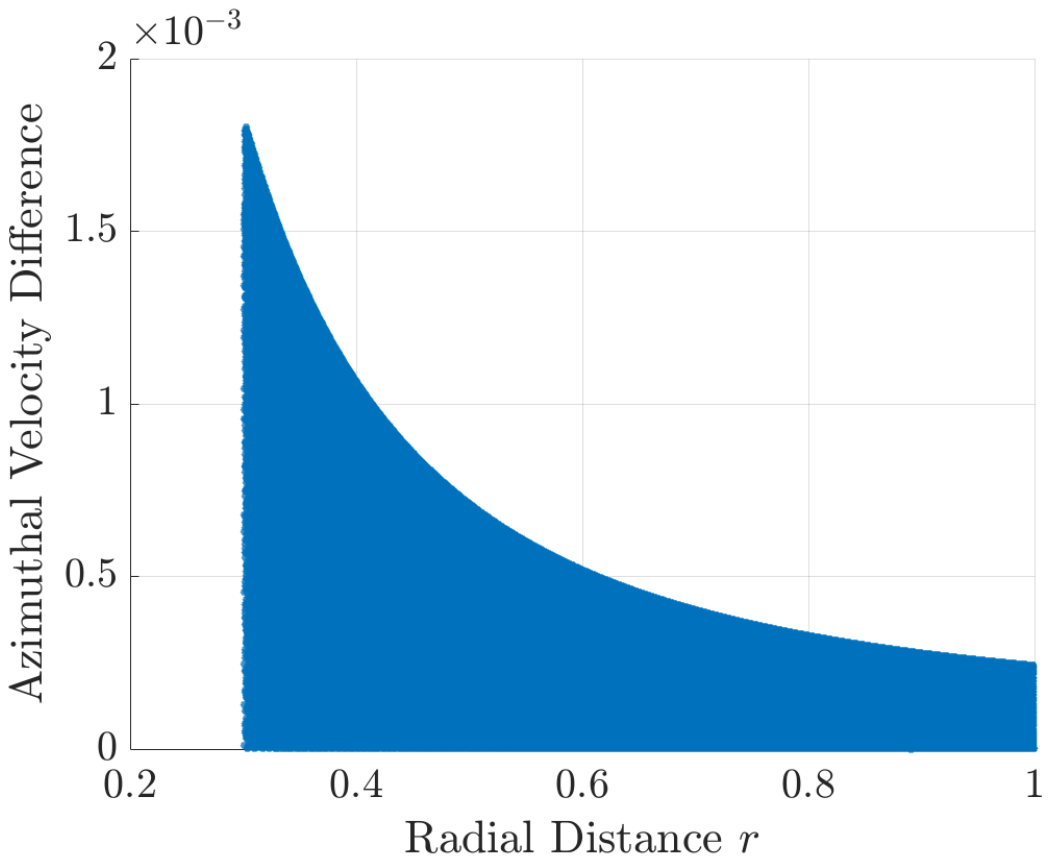}
        \caption{Azimuthal velocity difference}\label{fig:couette_difference}
    \end{subfigure}
    \caption{Taylor--Couette solution at $N=1023$: velocity field, computed azimuthal profile, and pointwise azimuthal-velocity error.}\label{fig:couette_flow}
\end{figure}

The concentric streamlines in Figure~\ref{fig:couette_flow}(a) show
that the Cartesian discretization does not introduce a visible preferred
direction. The computed values in panel (b) collapse onto the radial analytic
profile, while panel (c) localizes the remaining error near the inner unfitted
boundaries.

\begin{table}[htbp]
    \centering
    \small
    \begin{tabular}{
        @{} 
        S[table-format=4.0] 
        *{3}{S[table-format=1.4e-2] S[table-format=1.2]} 
        @{}
    }
        \toprule
        {$N$} & {$E_u$} & {Rate} & {$E_v$} & {Rate} & {$E_p$} & {Rate} \\
        \midrule
         127 & 5.2835e-05 & {--} & 5.2835e-05 & {--} & 1.2282e-03 & {--} \\
         255 & 1.3496e-05 & 1.97 & 1.3496e-05 & 1.97 & 4.1894e-04 & 1.55 \\
         511 & 3.4338e-06 & 1.97 & 3.4338e-06 & 1.97 & 1.1813e-04 & 1.83 \\
        1023 & 8.7357e-07 & 1.97 & 8.7358e-07 & 1.97 & 3.4001e-05 & 1.80 \\
        \bottomrule
    \end{tabular}
    \caption{Maximum-norm errors and convergence rates for the Taylor--Couette benchmark.}
    \label{tab:couette_errors}
\end{table}

Table~\ref{tab:couette_errors} gives rates of $1.97$ for both
velocity components at every refinement, demonstrating uniform second-order
velocity convergence. The pressure rate rises from $1.55$ to $1.80$ as the
grid is refined. Because the exact pressure is constant, this column tests
pressure recovery and gauge alignment rather than a nontrivial pressure
profile.



\begin{figure}[htbp]
\centering
    \begin{subfigure}{0.41\textwidth}
        \centering
        \includegraphics[width=\textwidth, trim = 2cm 7cm 2cm 6.5cm]{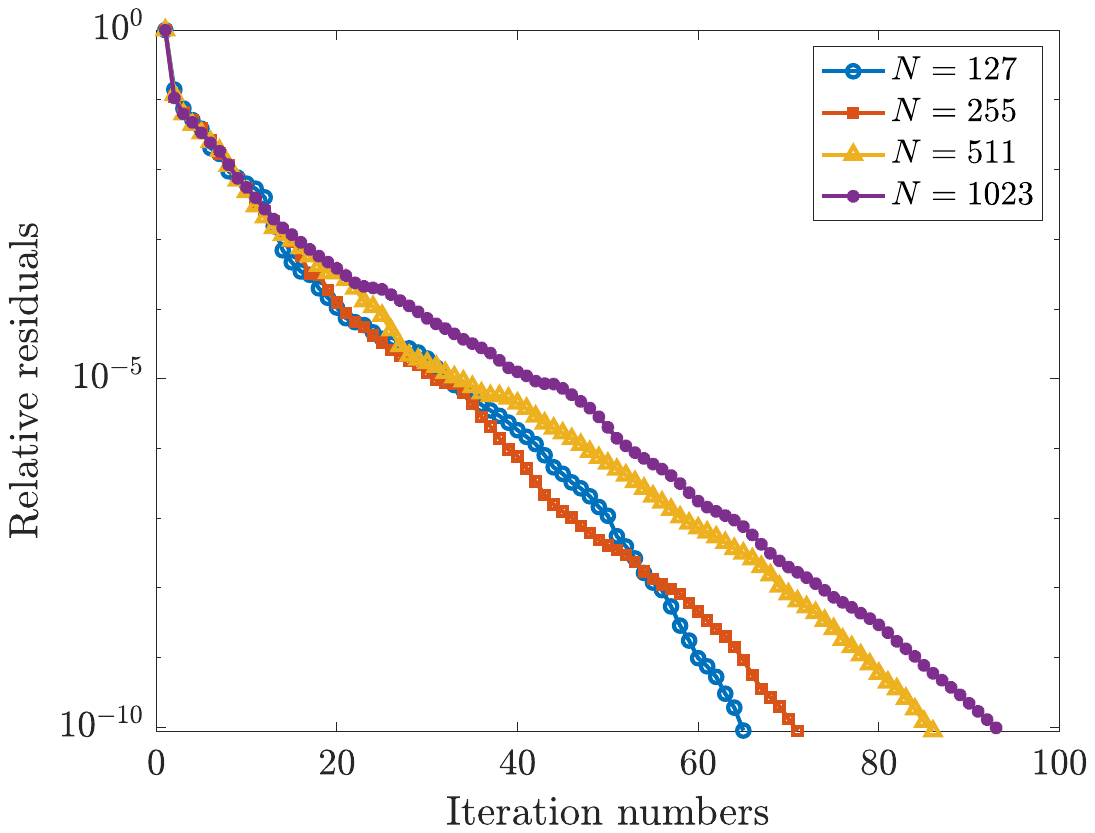}
    \caption{}
    \label{fig:couette_iter}
    \end{subfigure}
    ~
    \begin{subfigure}{0.45\textwidth}
        \centering
        \includegraphics[width=\textwidth, trim = 0cm 0cm 0cm 0cm]{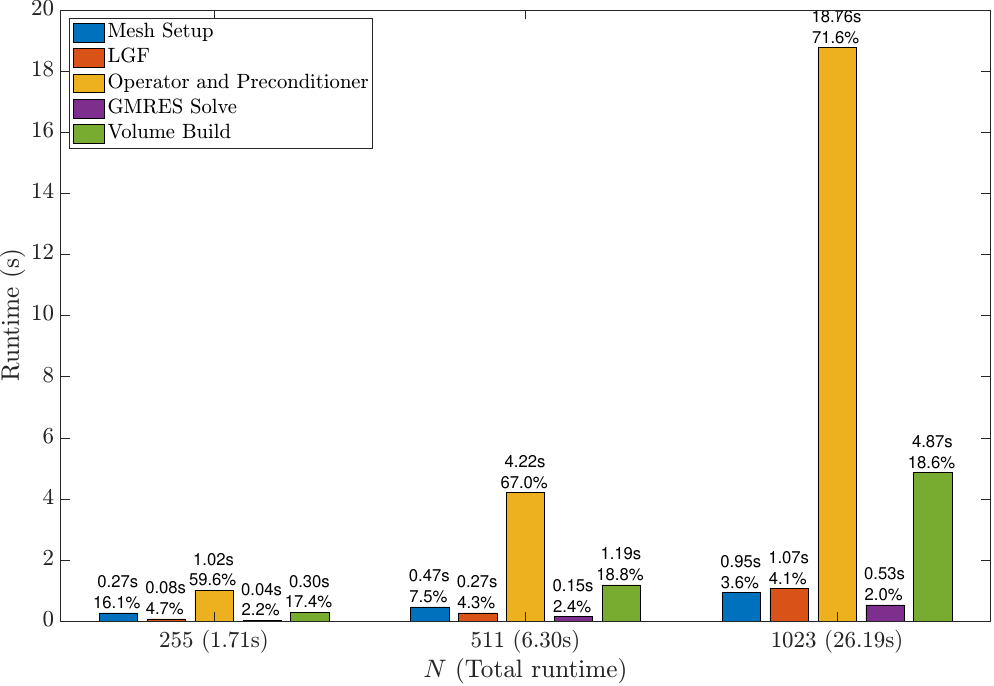}
        \caption{}
        \label{fig:couette_runtime}
    \end{subfigure}
    \caption{GMRES relative residuals and runtimes for the Taylor--Couette benchmark.}\label{fig:iterations_runtime}
\end{figure}

Figure~\ref{fig:iterations_runtime} presents the GMRES relative residuals; the uniformly exponential decay across meshes confirms the effectiveness of the preconditioner. The runtime breakdown shows that constructing the operator $A_c$ and the preconditioner $H_{\rm reg}$ takes approximately 60--72\% of total runtime. The total runtimes approximately quadruple upon mesh refinement from 1.71s to 6.30s and to 26.19s from $N=255$ to $N=1023$, consistent with near-optimal $\mathcal O(N^2\log N)$ scaling in the number of grid cells. We present residual histories and runtimes only for Taylor--Couette flow; the other cases are similar.

\begin{table}[htbp]
    \centering
    \small
    \begin{tabular}{
        @{} 
        S[table-format=4.0] 
        S[table-format=2.0]
        *{3}{S[table-format=1.4e-2]} 
        @{}
    }
        \toprule
        {$N$} & {Iterations} & {$\kappa(A_c)$} & {$\kappa(H_{\rm reg}A_c)$} & {$\nabla_h\cdot \bm{u}$} \\
        \midrule
         127 & 64 & 7.7788e4 & 2.0713e2 & 4.1455e-12 \\
         255 & 70 & 1.4381e5 & 2.0021e2 & 2.3489e-11 \\
         511 & 85 & 4.2946e5 & 2.2667e2 & 9.6010e-11 \\
        1023 & 92 & 1.4969e7 & 3.3102e3 & 5.6478e-10 \\
        \bottomrule
    \end{tabular}
    \caption{GMRES iterations, condition numbers before and after preconditioning, and maximum discrete divergence for Taylor--Couette flow.}
    \label{tab:couette_metric}
\end{table}

As shown in Table~\ref{tab:couette_metric}, the
condition number of $A_c$ grows from $7.8\times10^4$ to $1.5\times10^7$. The
Calder\'on preconditioner reduces these values substantially, although the
finest-grid preconditioned condition number still rises to $3.3\times10^3$. GMRES requires
64--92 iterations, and the maximum discrete divergence remains below
$5.7\times10^{-10}$. Thus the preconditioner controls the solve sufficiently
for the tested resolutions while the MAC incompressibility constraint is
retained to solver accuracy.


\subsubsection{Two-hole domain}

To test a multiply connected interior domain, we remove two circles
from the unit disk. Their centers are $(-0.4,0)$ and $(0.4,0.2)$, and their
radii are $0.3$ and $0.2$, respectively. Boundary data on all three components
are sampled from the force-doublet solution in \eqref{eq:doublet}.

\begin{table}[htbp]
    \centering
    \small
    \begin{tabular}{
        @{} 
        S[table-format=4.0] 
        *{3}{S[table-format=1.4e-2] S[table-format=1.2]} 
        @{}
    }
        \toprule
        {$N$} & {$E_u$} & {Rate} & {$E_v$} & {Rate} & {$E_p$} & {Rate} \\
        \midrule
         127 & 4.1571e-06 & {--} & 2.4229e-06 & {--} & 1.0489e-04 & {--} \\
         255 & 1.0589e-06 & 1.97 & 5.9806e-07 & 2.02 & 2.6692e-05 & 1.97 \\
         511 & 2.6720e-07 & 1.99 & 1.4972e-07 & 2.00 & 6.7469e-06 & 1.98 \\
        1023 & 6.7122e-08 & 1.99 & 3.7359e-08 & 2.00 & 1.7409e-06 & 1.95 \\
        \bottomrule
    \end{tabular}
    \caption{Maximum-norm errors and convergence rates for the force-doublet solution in the two-hole domain.}
    \label{tab:twoholes_errors}
\end{table}

\begin{table}[htbp]
    \centering
    \small
    \begin{tabular}{
        @{} 
        S[table-format=4.0] 
        S[table-format=3.0]
        *{3}{S[table-format=1.4e-2]} 
        @{}
    }
        \toprule
        {$N$} & {Iterations} & {$\kappa(A_c)$} & {$\kappa(H_{\rm reg}A_c)$} & {$\nabla_h\cdot \bm{u}$} \\
        \midrule
         127 & 122 & 8.8930e04 & 1.2676e03 & 1.1789e-11 \\
         255 & 113 & 1.6597e05 & 1.2691e03 & 5.4896e-11 \\
         511 & 153 & 1.2302e06 & 2.6562e03 & 4.5181e-10 \\
        1023 & 137 & 1.7261e07 & 5.5154e03 & 2.5178e-09 \\
        \bottomrule
    \end{tabular}
    \caption{GMRES iterations, condition numbers before and after preconditioning, and maximum discrete divergence in the two-hole domain.}
    \label{tab:twoholes_solver}
\end{table}

\begin{figure}[htbp]
\centering
    \begin{subfigure}{0.45\textwidth}
        \centering
        \includegraphics[width=\textwidth, trim = 2cm 7cm 2cm 6.5cm]{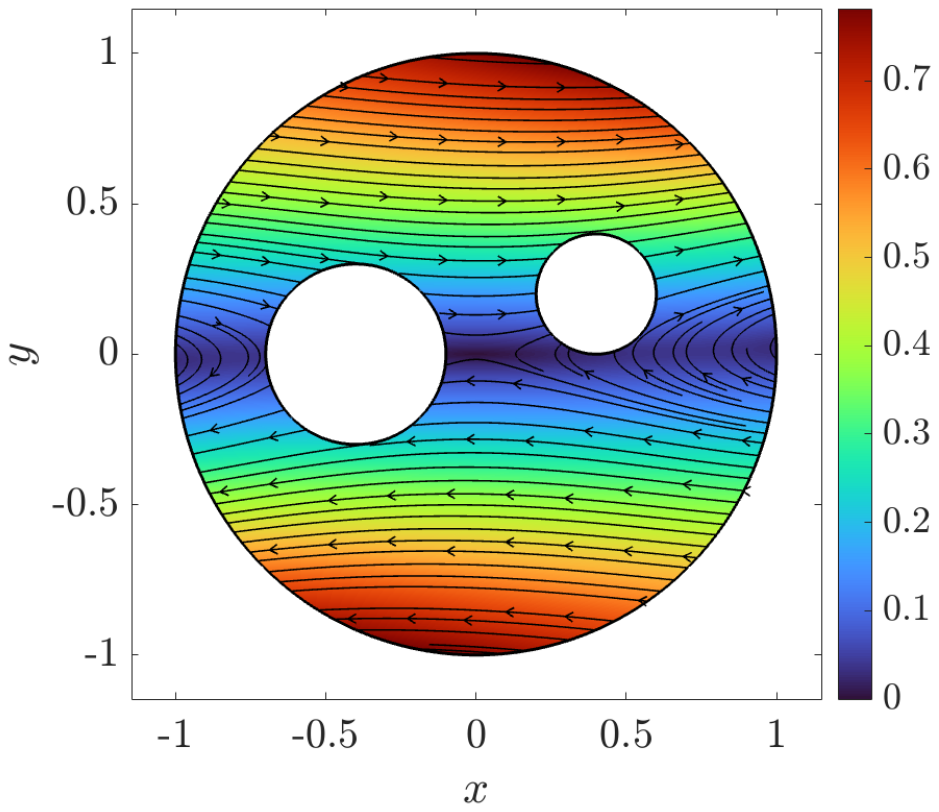}
        \caption{Velocity}\label{fig:twoholes_speed}
    \end{subfigure}
    ~
    \begin{subfigure}{0.45\textwidth}
        \centering
        \includegraphics[width=\textwidth, trim = 2cm 7cm 2cm 6.5cm]{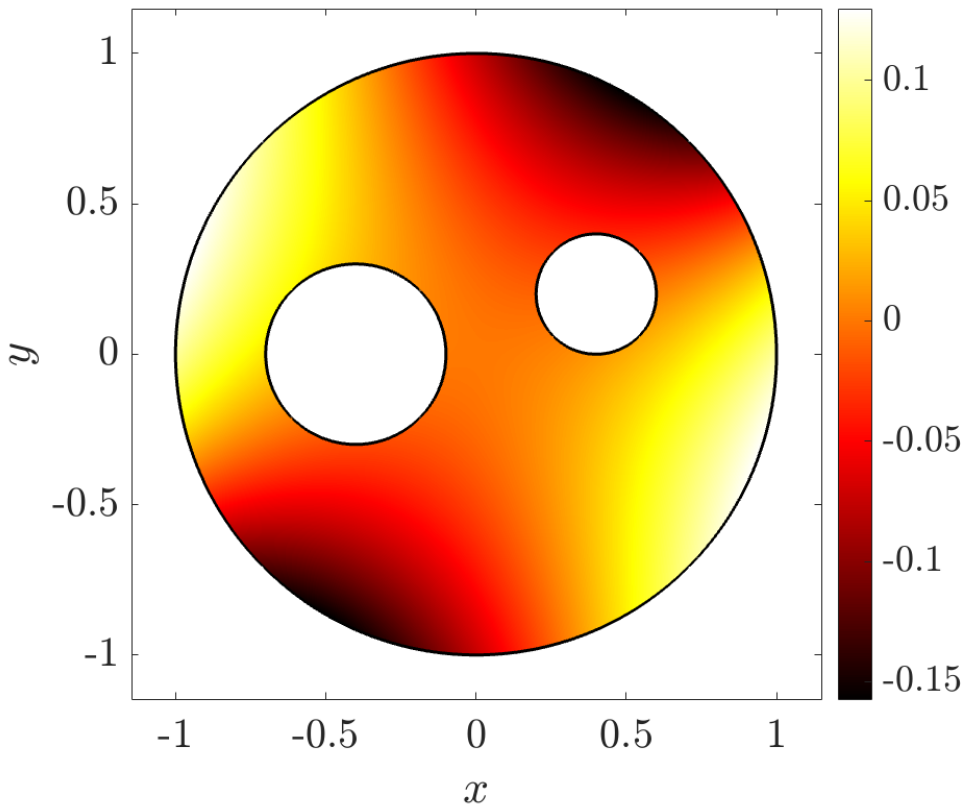}
        \caption{Pressure}\label{fig:twoholes_pressure}
    \end{subfigure}
    \caption{Velocity magnitude with streamlines and pressure in the two-hole domain at $N=1023$.}\label{fig:two_holes}
\end{figure}

Table~\ref{tab:twoholes_errors} shows second-order convergence for
both velocity components and for pressure despite the three boundary
components. Preconditioning reduces the condition number by factors ranging
from about $70$ on the coarsest grid to more than $3\times10^3$ on the
finest grid in Table~\ref{tab:twoholes_solver}. GMRES uses 113--153 iterations, and the maximum divergence stays
below $2.6\times10^{-9}$. Figure~\ref{fig:two_holes} confirms that the
solution remains smooth around both embedded boundaries.

\subsubsection{Narrow-gap three-component domain}
We next move the two circular holes close together. Their centers
are $(-0.3,0)$ and $(0.2,0)$ and their radii are $0.29$ and $0.19$, leaving a
gap of only $0.02$ between the boundaries. The outer boundary remains the unit
circle, and the force-doublet solution again supplies the data. This case
tests the sensitivity of the boundary closure and preconditioner to nearly
touching components.

\begin{table}[htbp]
    \centering
    \small
    \begin{tabular}{
        @{} 
        S[table-format=4.0] 
        *{3}{S[table-format=1.4e-2] S[table-format=1.2]} 
        @{}
    }
        \toprule
        {$N$} & {$E_u$} & {Rate} & {$E_v$} & {Rate} & {$E_p$} & {Rate} \\
        \midrule
         127 & 3.6647e-06 & {--} & 4.1043e-06 & {--} & 3.6862e-03 & {--} \\
         255 & 9.2874e-07 & 1.98 & 9.1154e-07 & 2.17 & 6.2074e-04 & 2.57 \\
         511 & 2.3410e-07 & 1.99 & 2.2908e-07 & 1.99 & 1.5822e-04 & 1.97 \\
        1023 & 5.8768e-08 & 1.99 & 5.6987e-08 & 2.01 & 3.9855e-05 & 1.99 \\
        \bottomrule
    \end{tabular}
    \caption{Maximum-norm errors and convergence rates for the force-doublet solution in the narrow-gap domain.}
    \label{tab:close_circles_errors}
\end{table}

\begin{table}[htbp]
    \centering
    \small
    \begin{tabular}{
        @{} 
        S[table-format=4.0] 
        S[table-format=3.0]
        *{3}{S[table-format=1.4e-2]} 
        @{}
    }
        \toprule
        {$N$} & {Iterations} & {$\kappa(A_c)$} & {$\kappa(H_{\rm reg}A_c)$} & {$\nabla_h\cdot \bm{u}$} \\
        \midrule
         127 & 141 & 2.6046e05 & 8.3943e03 & 1.1784e-11 \\
         255 & 117 & 1.6950e05 & 2.9591e03 & 5.4894e-11 \\
         511 & 167 & 6.4832e05 & 3.5393e03 & 4.5183e-10 \\
        1023 & 146 & 1.6974e07 & 5.5851e03 & 2.5179e-09 \\
        \bottomrule
    \end{tabular}
    \caption{GMRES iterations, condition numbers before and after preconditioning, and maximum discrete divergence in the narrow-gap domain.}
    \label{tab:close_circles_solver}
\end{table}

The two velocity errors in Table~\ref{tab:close_circles_errors}
retain second-order convergence. The pressure error is comparatively large on
the coarsest grid, for which the $0.02$ gap is only weakly resolved, but the
rate settles to approximately two on the last two refinements. The
preconditioned condition numbers remain between $3.0\times10^3$ and
$8.4\times10^3$, and GMRES takes 117--167 iterations in Table~\ref{tab:close_circles_solver}. Relative to the
separated-hole case, the narrow gap is therefore more demanding, but it does
not change the asymptotic accuracy. The divergence remains below
$2.6\times10^{-9}$.

\begin{figure}[htbp]
\centering
    \begin{subfigure}{0.45\textwidth}
        \centering
        \includegraphics[width=\textwidth, trim = 2cm 7cm 2cm 6.5cm]{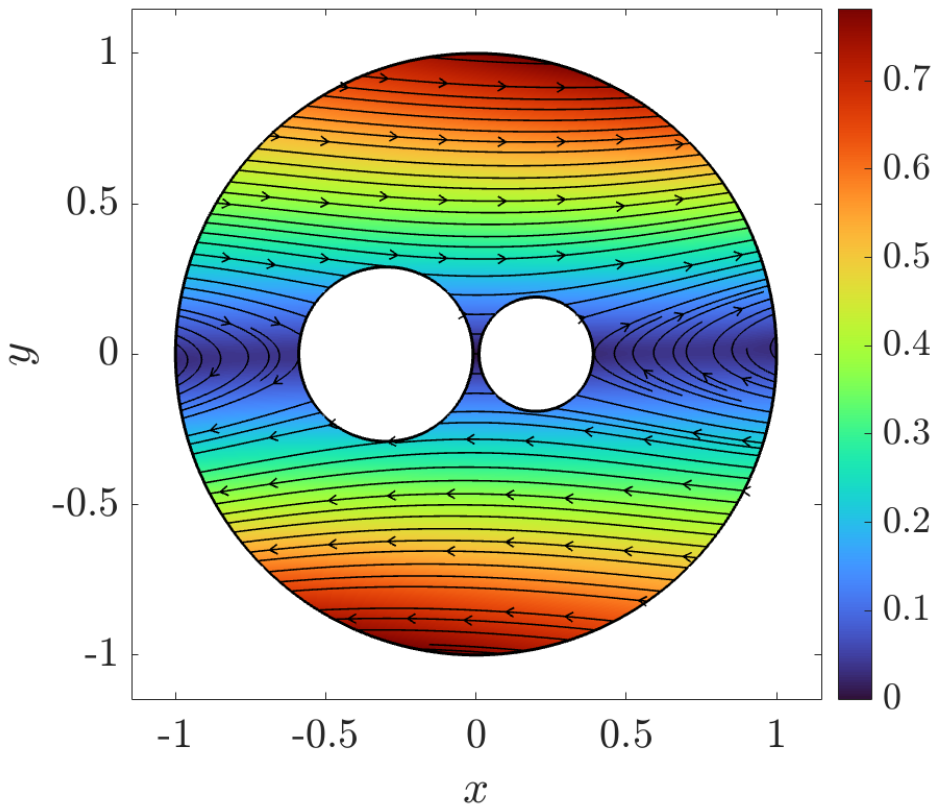}
        \caption{Velocity}\label{fig:close_speed}
    \end{subfigure}
    ~
    \begin{subfigure}{0.45\textwidth}
        \centering
        \includegraphics[width=\textwidth, trim = 2cm 7cm 2cm 6.5cm]{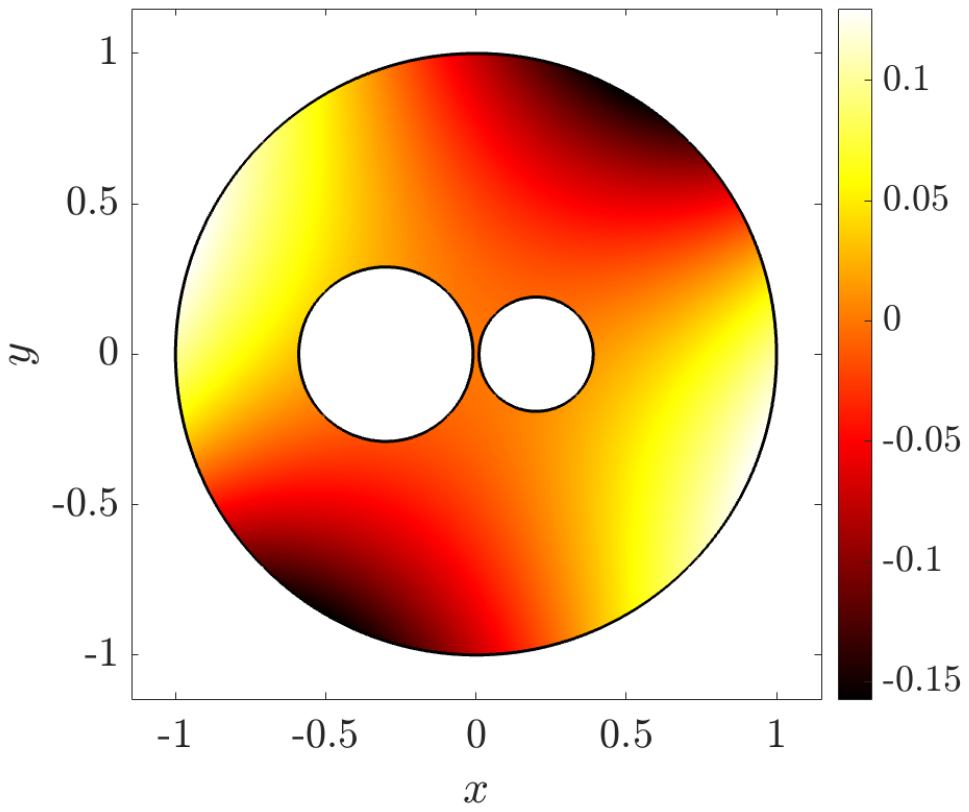}
        \caption{Pressure}\label{fig:close_pressure}
    \end{subfigure}
    \caption{Velocity magnitude with streamlines and pressure in the narrow-gap domain at $N=1023$.}\label{fig:close_circles}
\end{figure}

Figure~\ref{fig:close_circles} plots the approximated solution for the velocity magnitude with streamlines and pressure in a narrow-gap setup.

\subsubsection{\texorpdfstring{Moffatt eddies in a $30^\circ$ wedge}{Moffatt eddies in a 30-degree wedge}}
\label{sec:moffatt_qualitative}

This example tests whether the method can resolve localized flow
structures over several spatial scales. A remote disturbance in a sufficiently
sharp Stokes corner generates a sequence of counter-rotating Moffatt eddies
whose sizes and intensities decrease geometrically toward the apex
\citep{moffatt1964viscous}. For a wedge of total opening angle $30^\circ$
(half-angle $15^\circ$), the predicted asymptotic ratio of consecutive
eddy-center distances is $r_n/r_{n+1}\approx2.10$.

The wedge apex is at $y=-0.8$, and a tangentially moving curved
boundary drives the flow with unit tangent velocity. On the $N=1023$ MAC grid, the
preconditioned system has condition number $1.82\times10^3$ and GMRES
converges in 82 iterations. The maximum discrete divergence is
$1.45\times10^{-7}$.

\begin{figure}[htbp]
\centering
    \includegraphics[width=0.5\textwidth, trim = 2cm 7cm 2cm 6.5cm]{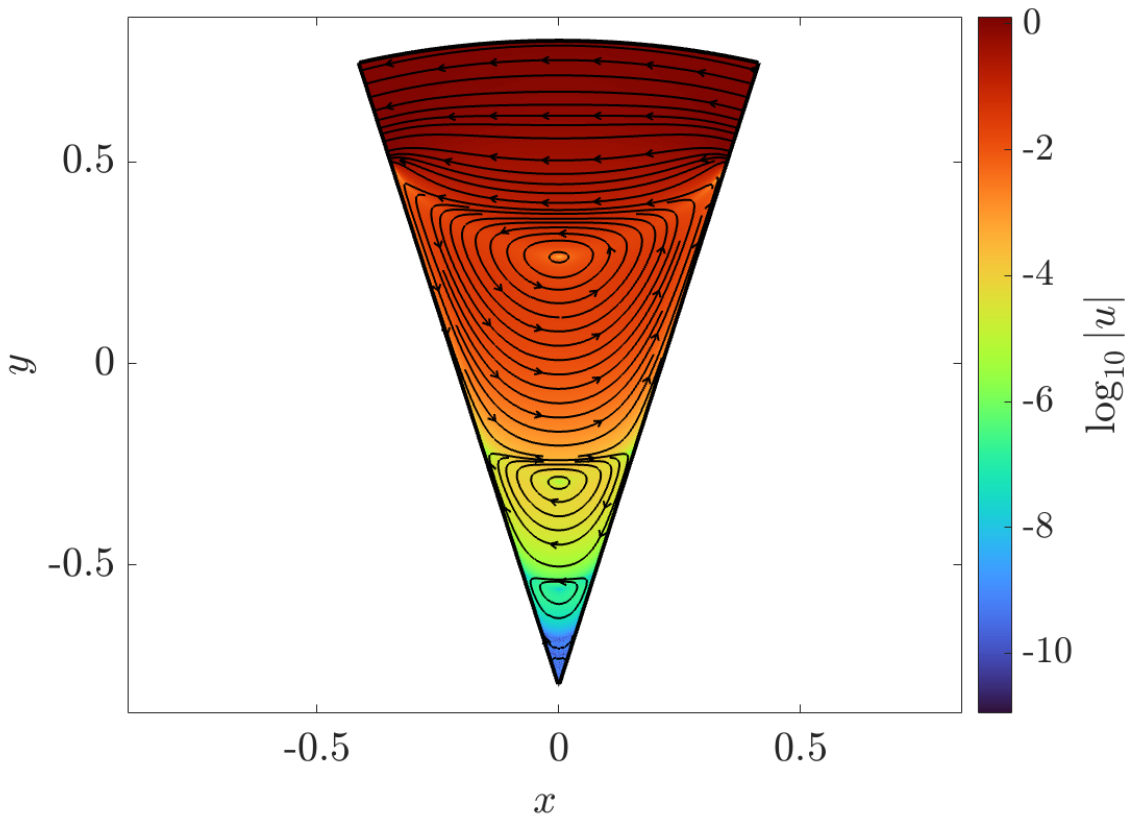}
    \caption{Moffatt eddies in a $30^\circ$ wedge at $N=1023$; color denotes $\log_{10}|\bm u|$.}
    \label{fig:moffatt}
\end{figure}

We locate the first four vortex centers from zero crossings of the
horizontal velocity, $u=0$, along the symmetry line $x=0$ in
Figure~\ref{fig:moffatt}. Their distances from the apex are
$r_1=0.899657$, $r_2=0.427323$, $r_3=0.202243$, and $r_4=0.096529$,
which give $r_1/r_2=2.105$, $r_2/r_3=2.113$, and $r_3/r_4=2.095$. A visualization is given in Figure~\ref{fig:moffatt}.

Thus the experiment resolves
three successive scale reductions accurately; a dedicated refinement study
would be needed to attribute the remaining discrepancy to a specific source.

\subsubsection{Interior flow with body forcing}

To verify the volume-potential construction of
Section~\ref{sec:inhomogeneous_forcing}, we return to the unit disk
($\alpha=1$ in the ellipse study) and prescribe the smooth exact fields
\begin{subequations}
\begin{align}
u &= \sin(\pi x)\cos(\pi y),\\
v &= -\cos(\pi x)\sin(\pi y),\\
p &= \cos(2\pi x)\sin(\pi y).
\end{align}
\end{subequations}
Substitution into the Stokes equations gives the body force
\begin{subequations}
\begin{align}
f^u &= 2\mu\pi^2\sin(\pi x)\cos(\pi y)-2\pi\sin(2\pi x)\sin(\pi y),\\
f^v &= -2\mu\pi^2\cos(\pi x)\sin(\pi y) + \pi\cos(2\pi x)\cos(\pi y).
\end{align}
\end{subequations}
The exact velocity is divergence-free, so this test isolates the
accuracy of the particular volume potential and its boundary correction.

\begin{table}[htbp]
    \centering
    \small
    \begin{tabular}{
        @{} 
        S[table-format=4.0] 
        *{3}{S[table-format=1.4e-2] S[table-format=1.2]} 
        @{}
    }
        \toprule
        {$N$} & {$E_u$} & {Rate} & {$E_v$} & {Rate} & {$E_p$} & {Rate} \\
        \midrule
         127 & 1.1108e-04 & {--} & 1.2326e-04 & {--} & 2.3026e-03 & {--} \\
         255 & 2.8287e-05 & 1.97 & 3.1289e-05 & 1.98 & 6.0868e-04 & 1.92 \\
         511 & 7.1008e-06 & 1.99 & 7.8491e-06 & 2.00 & 1.5283e-04 & 1.99 \\
        1023 & 1.7822e-06 & 1.99 & 1.9687e-06 & 2.00 & 3.8318e-05 & 2.00 \\
        \bottomrule
    \end{tabular}
    \caption{Maximum-norm errors and convergence rates for the body-forced trigonometric solution.}
    \label{tab:trig_errors}
\end{table}

\begin{table}[htbp]
    \centering
    \small
    \begin{tabular}{
        @{} 
        S[table-format=4.0] 
        S[table-format=2.0]
        *{3}{S[table-format=1.4e-2]} 
        @{}
    }
        \toprule
        {$N$} & {Iterations} & {$\kappa(A_c)$} & {$\kappa(H_{\rm reg}A_c)$} & {$\nabla_h\cdot \bm{u}$} \\
        \midrule
         127 & 68 & 6.0599e4 & 1.2440e2 & 2.8241e-11 \\
         255 & 59 & 9.0113e4 & 7.7860e1 & 1.4377e-10 \\
         511 & 84 & 3.3379e5 & 1.4656e2 & 8.2074e-10 \\
        1023 & 94 & 1.1605e7 & 2.2004e3 & 9.3361e-09 \\
        \bottomrule
    \end{tabular}
    \caption{GMRES iterations, condition numbers before and after preconditioning, and maximum discrete divergence for the body-forced solution.}
    \label{tab:trig_solver}
\end{table}

All three fields converge at second order in
Table~\ref{tab:trig_errors}. Because the boundary operator is unchanged by
the volume forcing, the condition numbers in Table~\ref{tab:trig_solver}
match those of the same unit-disk geometry in the homogeneous test. GMRES
takes 59--94 iterations, and the maximum divergence remains below
$9.4\times10^{-9}$. These results confirm that adding a volume potential
preserves both the boundary convergence rate and the discrete incompressibility
property.

\subsection{Exterior flows}
In this subsection, we consider two cases of exterior flows. The fluid occupies the unbounded complement of the prescribed obstacles. Dirichlet data are imposed only on the obstacle boundaries; the free-space LGF representation requires no artificial boundary conditions on any outer truncation of the computational box $[-1.15,1.15]^2$, which is used solely for evaluation and error measurement. Because the Stokes LGF grows logarithmically, a force-balanced density (Lemma~\ref{lem:stokes-single-layer-energy}) is essential for a bounded velocity at infinity; the hydrostatic rank completion of Section~\ref{sec:boundary-closure} places the discrete density in this compatible class.
\subsubsection{Single-obstacle manufactured solution}
A circular obstacle of radius $0.3$ is placed at the origin. The two
opposing Stokeslets are located at $(0.05,0.05)$ and $(-0.05,-0.05)$,
inside the excluded circle, with forces $(10,1)^T$ and $(-10,-1)^T$.
The force-doublet field is therefore smooth throughout the fluid region and
provides exact Dirichlet data on the obstacle.

\begin{table}[htbp]
    \centering
    \small
    \begin{tabular}{
        @{} 
        S[table-format=4.0] 
        *{3}{S[table-format=1.4e-2] S[table-format=1.2]} 
        @{}
    }
        \toprule
        {$N$} & {$E_u$} & {Rate} & {$E_v$} & {Rate} & {$E_p$} & {Rate} \\
        \midrule
         127 & 1.5550e-04 & {--} & 1.0101e-04 & {--} & 3.5642e-03 & {--} \\
         255 & 4.0325e-05 & 1.95 & 2.5936e-05 & 1.96 & 1.3612e-03 & 1.39 \\
         511 & 1.0249e-05 & 1.98 & 6.5884e-06 & 1.98 & 2.9871e-04 & 2.19 \\
        1023 & 2.5797e-06 & 1.99 & 1.6695e-06 & 1.98 & 9.4356e-05 & 1.66 \\
        \bottomrule
    \end{tabular}
    \caption{Maximum-norm errors and convergence rates for the single-obstacle exterior solution.}
    \label{tab:exterior_errors}
\end{table}

\begin{table}[htbp]
    \centering
    \small
    \begin{tabular}{
        @{} 
        S[table-format=4.0] 
        S[table-format=2.0]
        *{3}{S[table-format=1.4e-2]} 
        @{}
    }
        \toprule
        {$N$} & {Iterations} & {$\kappa(A_c)$} & {$\kappa(H_{\rm reg}A_c)$} & {$\nabla_h\cdot \bm{u}$} \\
        \midrule
         127 & 42 & 6.5248e03 & 5.0344e01 & 1.9231e-11 \\
         255 & 62 & 2.7607e04 & 1.1000e02 & 1.1575e-10 \\
         511 & 54 & 4.3124e04 & 6.3385e01 & 6.5885e-10 \\
        1023 & 73 & 1.5279e05 & 1.1768e02 & 3.2557e-09 \\
        \bottomrule
    \end{tabular}
    \caption{GMRES iterations, condition numbers before and after preconditioning, and maximum discrete divergence for the single-obstacle exterior solution.}
    \label{tab:exterior_solver}
\end{table}

The two velocity components converge at rates between $1.95$ and
$1.99$ (Table~\ref{tab:exterior_errors}). The pressure rate varies more on
individual refinements, but the total reduction from $N=127$ to $N=1023$
corresponds to an average order of approximately $1.75$. In
Table~\ref{tab:exterior_solver}, preconditioning keeps the condition number
between $50$ and $118$ and GMRES converges in 42--73 iterations, while the
maximum divergence stays below $3.3\times10^{-9}$.

\begin{figure}[htbp]
\centering
    \begin{subfigure}{0.45\textwidth}
        \centering
        \includegraphics[width=\textwidth, trim = 2cm 7cm 2cm 6.5cm]{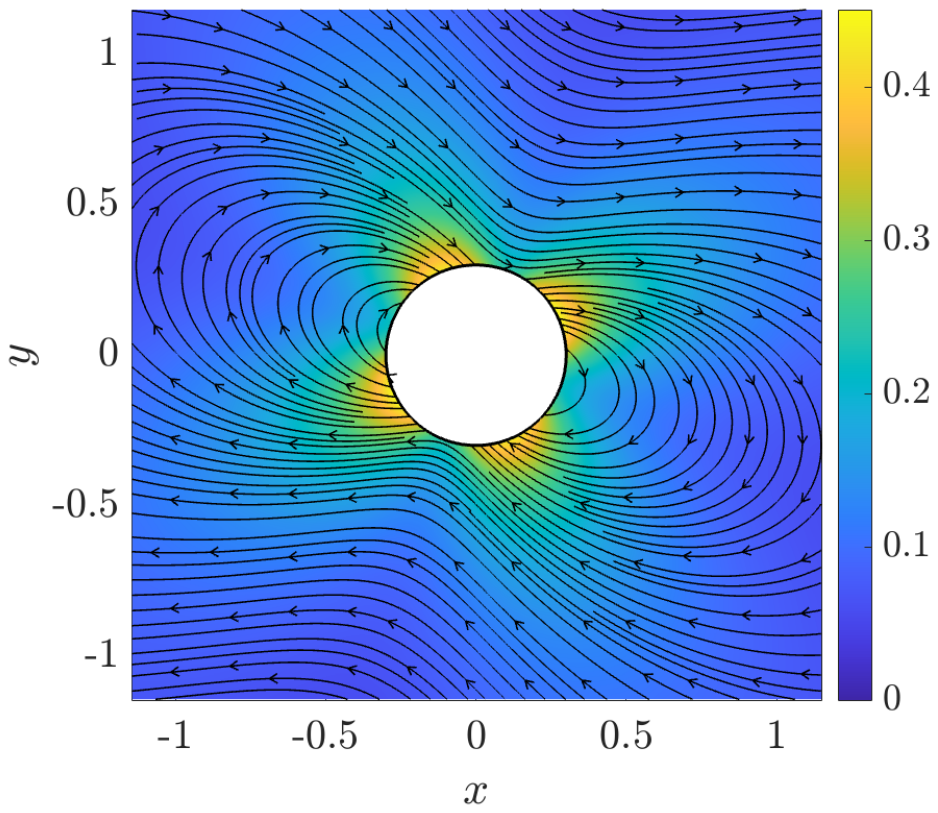}
        \caption{Velocity}\label{fig:ex_speed}
    \end{subfigure}
    ~
    \begin{subfigure}{0.45\textwidth}
        \centering
        \includegraphics[width=\textwidth, trim = 2cm 7cm 2cm 6.5cm]{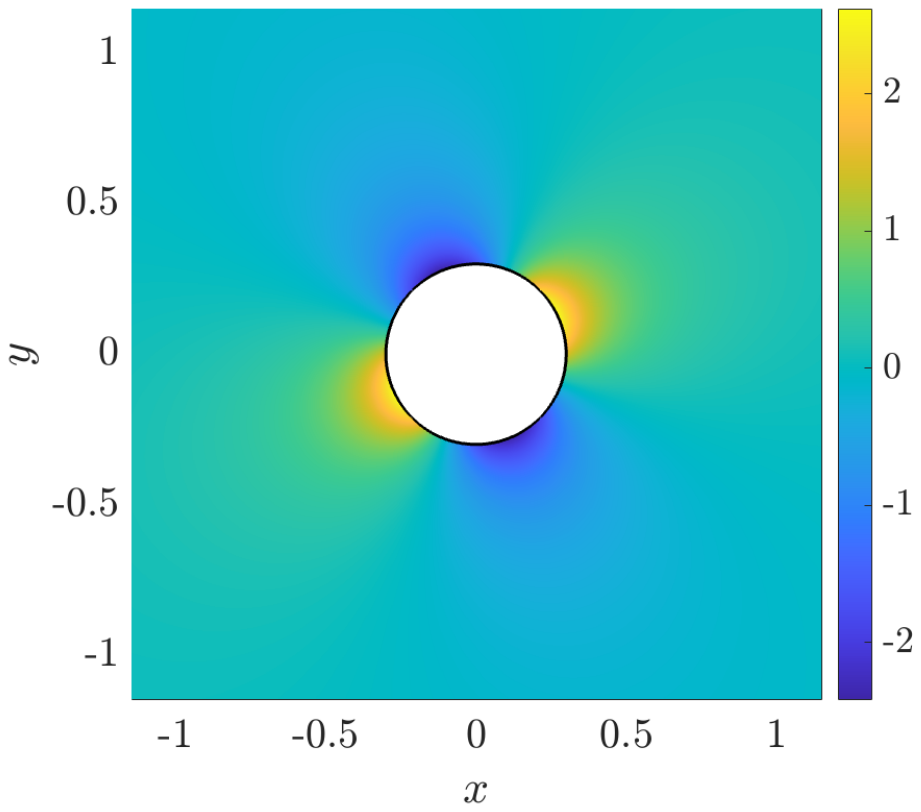}
        \caption{Pressure}\label{fig:ex_pressure}
    \end{subfigure}
    \caption{Velocity magnitude with streamlines and pressure for the single-obstacle exterior flow at $N=1023$.}\label{fig:exterior_flow}
\end{figure}

The approximated velocity magnitude with streamlines and pressure on mesh $N=1023$ are presented in Figure~\ref{fig:exterior_flow}.

\subsubsection{Three-obstacle manufactured solution}

The final manufactured example contains three elliptical obstacles.
Their boundaries are the zero level sets of
\begin{subequations}
\begin{align}
    \phi_1(x,y) &= 1 - \frac{(x+0.3)^2}{0.15^2} - \frac{(y-0.2)^2}{0.25^2}, \\
    \phi_2(x,y) &= 1 - \frac{(x-0.2)^2}{0.25^2} - \frac{(y+0.2)^2}{0.1^2}, \\
    \phi_3(x,y) &= 1 - \frac{(x-0.1)^2}{0.24^2} - \frac{(y-0.35)^2}{0.12^2}.
\end{align}
\end{subequations}
The obstacle centers are $\bm x_{c1}=(-0.3,0.2)$,
$\bm x_{c2}=(0.2,-0.2)$, and $\bm x_{c3}=(0.1,0.35)$.

We construct a smooth exterior solution by placing Stokes
singularities at these centers, which lie outside the fluid domain, and set
$\mu=1$.

For a source at $\bm x_0$, let $\bm r=\bm x-\bm x_0$ and
$r=|\bm r|$. A two-dimensional rotlet of torque $L$ generates
\begin{equation}
    \bm{u}_{\text{rot}}(\bm{x}; \bm{x}_0, L) = \frac{L}{4\pi\mu r^2} \begin{pmatrix} -r_y \\ r_x \end{pmatrix}, \quad p_{\text{rot}} = 0.
\end{equation}
A symmetric, torque-free stresslet with force $\bm f$ and dipole
direction $\bm d$ generates
\begin{subequations}
\begin{align}
    \bm{u}_{\text{str}}(\bm{x}; \bm{x}_0, \bm{f}, \bm{d}) &= \frac{1}{4\pi\mu} \left( -\frac{(\bm{f} \cdot \bm{d})}{r^2}\bm{r} + \frac{2(\bm{r} \cdot \bm{f})(\bm{r} \cdot \bm{d})}{r^4}\bm{r} \right), \\
    p_{\text{str}}(\bm{x}; \bm{x}_0, \bm{f}, \bm{d}) &= \frac{1}{2\pi} \left( -\frac{\bm{f} \cdot \bm{d}}{r^2} + \frac{2(\bm{r} \cdot \bm{f})(\bm{r} \cdot \bm{d})}{r^4} \right).
\end{align}
\end{subequations}

The exact solution is the sum of two rotlets and one stresslet:
\begin{subequations}
\begin{align}
    \bm{u}(\bm{x}) &= \bm{u}_{\text{rot}}(\bm{x}; \bm{x}_{c1}, L_1) + \bm{u}_{\text{rot}}(\bm{x}; \bm{x}_{c2}, L_2) + \bm{u}_{\text{str}}(\bm{x}; \bm{x}_{c3}, \bm{f}, \bm{d}), \\
    p(\bm{x}) &= p_{\text{str}}(\bm{x}; \bm{x}_{c3}, \bm{f}, \bm{d}).
\end{align}
\end{subequations}
The rotlets at $\bm x_{c1}$ and $\bm x_{c2}$ have torques
$L_1=L_2=1$. The stresslet at $\bm x_{c3}$ uses
$\bm f=(1,0.3)^T$ and $\bm d=(0.4,1)^T$. This configuration tests the
rank completion for several disconnected boundaries as well as interactions
among distinct tensorial singularities.

\begin{figure}[htbp]
\centering
    \begin{subfigure}{0.45\textwidth}
        \centering
        \includegraphics[width=\textwidth, trim = 2cm 7cm 2cm 6.5cm]{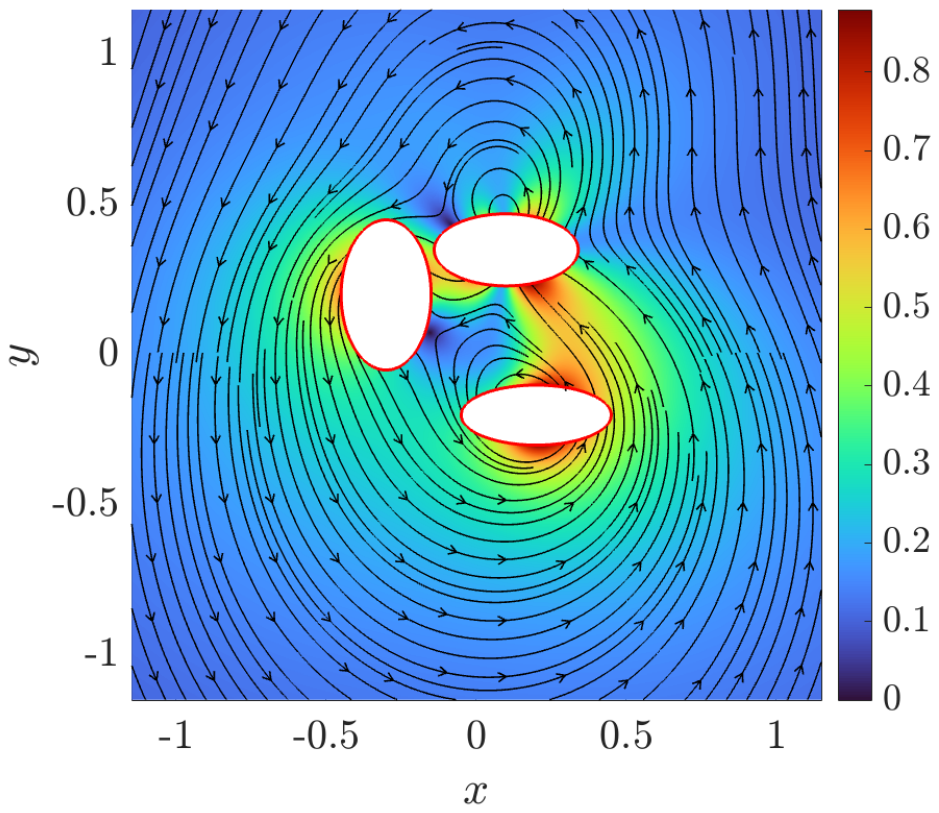}
        \caption{Velocity}\label{fig:multi_speed}
    \end{subfigure}
    ~
    \begin{subfigure}{0.45\textwidth}
        \centering
        \includegraphics[width=\textwidth, trim = 2cm 7cm 2cm 6.5cm]{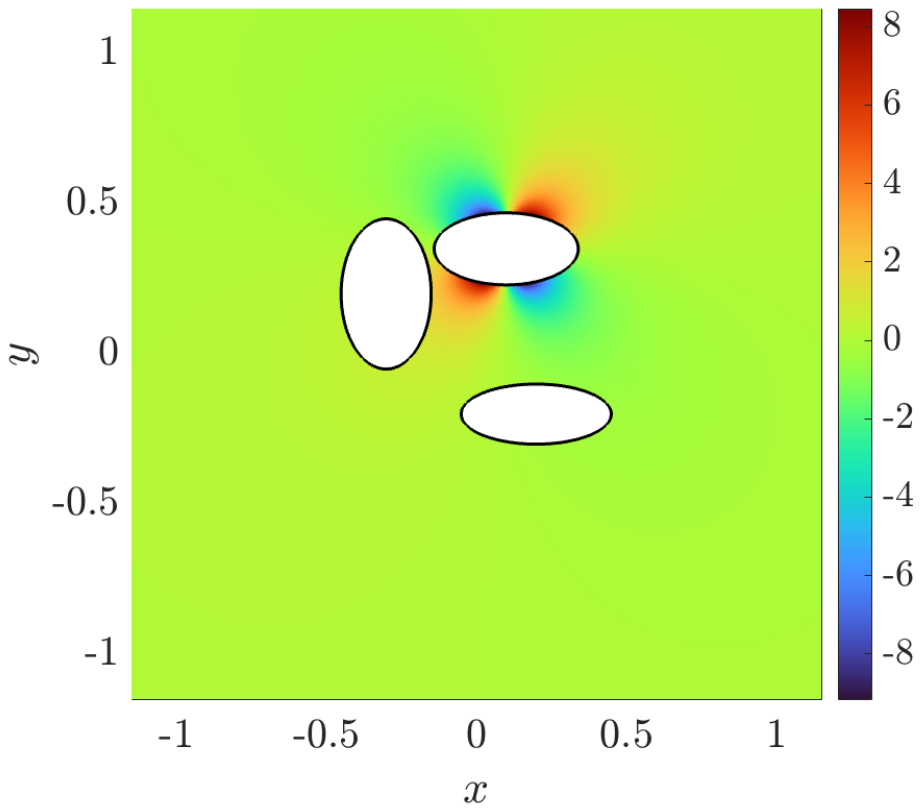}
        \caption{Pressure}\label{fig:multi_pressure}
    \end{subfigure}
    \caption{Velocity magnitude with streamlines and pressure for the three-obstacle exterior flow at $N=2047$.}\label{fig:multi_obs}
\end{figure}

\begin{table}[htbp]
    \centering
    \small
    \begin{tabular}{
        @{} 
        S[table-format=4.0] 
        *{3}{S[table-format=1.4e-2] S[table-format=1.2]} 
        @{}
    }
        \toprule
        {$N$} & {$E_u$} & {Rate} & {$E_v$} & {Rate} & {$E_p$} & {Rate} \\
        \midrule
         255 & 2.9714e-04 & {--} & 3.8773e-04 & {--} & 1.7405e-02 & {--} \\
         511 & 8.0391e-05 & 1.89 & 9.9531e-05 & 1.96 & 4.6968e-03 & 1.89 \\
        1023 & 2.1036e-05 & 1.93 & 2.5571e-05 & 1.96 & 1.2587e-03 & 1.90 \\
        2047 & 5.3415e-06 & 1.98 & 6.4262e-06 & 1.99 & 3.3151e-04 & 1.92 \\
        \bottomrule
    \end{tabular}
    \caption{Maximum-norm errors and convergence rates for the three-obstacle manufactured solution.}
    \label{tab:combined_ext_errors}
\end{table}

\begin{table}[htbp]
    \centering
    \small
    \begin{tabular}{
        @{} 
        S[table-format=4.0] 
        S[table-format=3.0]
        *{3}{S[table-format=1.4e-2]} 
        @{}
    }
        \toprule
        {$N$} & {Iterations} & {$\kappa(A_c)$} & {$\kappa(H_{\rm reg}A_c)$} & {$\nabla_h\cdot \bm{u}$} \\
        \midrule
         255 & 107 & 6.1632e04 & 6.6705e02 & 2.1988e-10 \\
         511 & 118 & 1.5156e05 & 5.7229e02 & 1.5983e-09 \\
        1023 & 132 & 5.4677e05 & 6.0863e02 & 6.0880e-09 \\
        2047 & 148 & 1.1876e06 & 6.6029e02 & 3.2678e-08 \\
        \bottomrule
    \end{tabular}
    \caption{GMRES iterations, condition numbers before and after preconditioning, and maximum discrete divergence for the three-obstacle solution.}
    \label{tab:combined_ext_solver}
\end{table}

Table~\ref{tab:combined_ext_errors} shows convergence rates between
$1.89$ and $1.99$ for all fields, with the rates moving closer to two as the
grid is refined. More importantly for the multiple-object completion, the
preconditioned condition number in Table~\ref{tab:combined_ext_solver} remains
between $5.7\times10^2$ and $6.7\times10^2$ over three refinements. GMRES
iterations increase gradually from 107 to 148, and the maximum divergence is
below $3.3\times10^{-8}$. The fields in Figure~\ref{fig:multi_obs} remain
smooth in the inter-obstacle regions and show no visible grid-aligned artifacts
near any boundary component.

\subsection{Overall observations}

Across the manufactured tests, the velocity converges uniformly at
approximately second order. Pressure is also close to second order in the
ellipse, multiply connected, body-forced, and three-obstacle tests; the
Taylor--Couette and single-obstacle exterior cases show more variable pressure
rates but monotone error reduction. The maximum discrete divergence ranges
from $10^{-12}$ to $10^{-7}$, confirming that geometry interpolation does not
destroy the MAC incompressibility constraint. Finally, the Calder\'on
preconditioner reduces the condition number by orders of magnitude in every
test. Its strongest nearly mesh-independent behavior occurs in the exterior
three-obstacle example, while the narrow-gap and finest-grid interior cases
remain the most demanding.

\section{Conclusion}
\label{sec:conclusion}

We have developed an unfitted boundary algebraic equation method for
the two-dimensional steady incompressible Stokes equations on a staggered MAC
grid. Regularized Laplace and biharmonic lattice Green's functions generate the
discrete Stokes velocity and pressure kernels, while local cut-point
interpolation imposes Dirichlet data without requiring a boundary-fitted mesh.
Sampled-normal rank updates remove the hydrostatic null modes, including the
independent modes associated with multiple obstacles. Nonzero body forces are
incorporated through a particular volume potential, leaving the boundary
operator unchanged, while padded FFTs accelerate the evaluation of both volume
and boundary convolutions.

Numerical experiments verify both the accuracy and the structural
properties of the method. Velocity converges at approximately second order
across the manufactured tests, with pressure converging at a similar rate in
most configurations. The maximum discrete divergence ranges from $10^{-12}$ to
$10^{-7}$, indicating that the unfitted boundary closure preserves the
MAC-grid incompressibility constraint to the accuracy of the iterative solve
and floating-point evaluation. By counteracting the mesh-dependent spectral
growth of the boundary system, the discrete Calder\'on construction reduces
the condition number by orders of magnitude and substantially improves solver
robustness for multiply connected and multiple-obstacle geometries, with the strongest conditioning gains observed for exterior flows. The method
also resolves complex multiscale flow structures: the Moffatt-eddy calculation,
for example, recovers the first three asymptotic scale ratios to within $0.7\%$.

The principal mathematical attraction of the framework is its
separation of tasks: the LGFs on the background grid enforce the discrete Stokes
equations and incompressibility, while geometry dependence is confined to thin
boundary layers and local interpolation. Its distinctive contribution
is this combination of a boundary-only density solve, sharp cut-point geometry
enforcement, MAC-grid incompressibility, hydrostatic rank completion, and
discrete Calder\'on preconditioning. In this sense, the method is a
discretize-then-represent boundary-integral analogue built from the MAC operator
itself, rather than a quadrature discretization of the continuous Stokes BIE.
It makes complex geometries easier to
accommodate without a fitted mesh or singular surface quadrature; the boundary
operator can be reused when only the body force changes, and the lattice kernels
can be reused when the geometry changes on a fixed grid.

While the current boundary algebraic framework effectively resolves complex geometries, several natural directions remain for further development. Close gaps and corners may benefit from adaptive or locally refined lattices, while larger boundary systems will require stronger matrix-free or multilevel acceleration. To further enhance solver robustness and expand the algebraic representation, developing a discrete double-layer formulation for the Stokes equations is a logical next step. Furthermore, this discrete potential framework can be naturally generalized to solid mechanics, specifically to boundary algebraic equations for linear elasticity, building on LGF-based discrete boundary-element ideas already explored in atomistic/continuum coupling \citep{hodapp2019lattice}. Finally, the volume-potential and homogeneous-correction decomposition provides a natural organizing principle for extensions to three-dimensional, time-dependent, and incompressible Navier--Stokes problems.

\section*{Acknowledgement}

This work was partially supported by the National Natural Science Foundation of China (Grant Nos. 12471342 to W. Ying and 12401546 to Q. Xia). Q. Xia additionally acknowledges funding from Wenzhou-Kean University (Grant Nos. ISRG2024003 and KY20250604000452). Q. Xia also thanks Anna-Karin Tornberg for many inspiring and helpful discussions on boundary integral method for Stokes flow.

\bibliographystyle{unsrtnat}
\bibliography{refs}

\end{document}